\documentclass[a4paper,12pt]{amsart}
\usepackage{latexsym,amscd,amssymb,url}
\usepackage[all,cmtip]{xy}  
\usepackage{graphicx}
\usepackage{xcolor}
\usepackage{enumerate} 
\definecolor{dblue}{rgb}{0,0,.6}
\usepackage[colorlinks=true, linkcolor=dblue, citecolor=dblue, filecolor = dblue, menucolor = dblue, urlcolor = dblue]{hyperref}

\numberwithin{equation}{section}

\newtheorem{theorem}{Theorem}[section]

\theoremstyle{plain}

\newtheorem{question}[theorem]{Question}

\newtheorem{corollary}[theorem]{Corollary}

\newtheorem{definition}[theorem]{Definition}

\newtheorem{lemma}[theorem]{Lemma}
\newtheorem{notation}[theorem]{Notation}

\newtheorem{proposition}[theorem]{Proposition}
\newtheorem{remark}[theorem]{Remark}

\newcommand{\del}{\partial}
\newcommand{\Z}{\mathbb Z}

\newcommand{\C}{\mathbb C}
\newcommand{\N}{\mathbb N}
\newcommand{\R}{\mathbb R}

\newcommand{\CP}{\mathbb P}

\newcommand{\RHom}{\operatorname{RHom}}

\newcommand{\im}{\operatorname{im}}

\newcommand{\Hom}{\operatorname{Hom}}
\newcommand{\Pic}{\operatorname{Pic}}
\newcommand{\Div}{\operatorname{Div}}

\newcommand{\Sym}{\operatorname{Sym}}
\newcommand{\id}{\operatorname{id}}

\newcommand{\Spec}{\operatorname{Spec}}

\newcommand{\pr}{\operatorname{pr}}

\newcommand{\codim}{\operatorname{codim}}

\newcommand{\Ext}{\operatorname{Ext}}

\newcommand{\CH}{\operatorname{CH}}
\newcommand{\supp}{\operatorname{supp}}
\newcommand{\sing}{\operatorname{sing}} 
\newcommand{\red}{\operatorname{red}}

\newcommand{\Frac}{\operatorname{Frac}}

 \newcommand{\sm}{\operatorname{sm}}

  \newcommand{\Tor}{\operatorname{Tor}}
  
  \newcommand{\A}{\mathbb{A}}
  \newcommand{\ord}{\operatorname{ord}}

\newcommand{\dashedlongrightarrow}{\xymatrix@1@=15pt{\ar@{-->}[r]&}}
\renewcommand{\longrightarrow}{\xymatrix@1@=15pt{\ar[r]&}}
\renewcommand{\mapsto}{\xymatrix@1@=15pt{\ar@{|->}[r]&}}
\renewcommand{\twoheadrightarrow}{\xymatrix@1@=15pt{\ar@{->>}[r]&}}
\newcommand{\hooklongrightarrow}{\xymatrix@1@=15pt{\ar@{^(->}[r]&}}
\newcommand{\congpf}{\xymatrix@1@=15pt{\ar[r]^-\sim&}}
\renewcommand{\cong}{\simeq}

\begin{document}    

\title[Unramified cohomology, algebraic cycles and rationality]{Unramified cohomology, algebraic cycles and rationality}

\author{Stefan Schreieder} 
\address{Institute of Algebraic Geometry, Leibniz University Hannover, Welfengarten 1, 30167 Hannover , Germany.}
\email{schreieder@math.uni-hannover.de}

\date{February 18, 2021} 
%


\begin{abstract}    
This is a survey on unramified cohomology with a view towards its applications to rationality problems.
\end{abstract}

\maketitle


\section{Introduction}

A variety $X$ of dimension $n$ over a field $k$ is said to be rational if it is birational to $\CP_k^n$, which means that it becomes isomorphic to $\CP_k^n$ after removing proper closed subsets from both sides.
More generally, $X$ is said to be stably rational if $X\times \CP^m_k$ is rational for some $m\geq 0$.

Generic projection shows that any variety $X$ of dimension $n$ over a field $k$ is birational to a hypersurface $\{F=0\}\subset \CP^{n+1}_k$.
The question whether $X$ is rational is then equivalent to asking whether over the algebraic closure of $k$, almost all solutions of the equation $F=0$  can be parametrized $1:1$  by a parameter $t\in \A^n_k$ (via rational functions with coefficients in $k$).
Rationality of $X$ thus translates into a very basic question about the solutions of $F=0$.

Disproving rationality for a given (rationally connected) variety $X$ is in general a subtle problem, which requires the computation of invariants that allow to distinguish irrational varieties from those that are rational.
Arguably the most powerful such invariant is unramified cohomology, which has its roots in the work of Bloch--Ogus \cite{BO} and which was introduced into the subject by Colliot-Thélène--Ojanguren \cite{CTO}.
Unramified cohomology is a generalization of the torsion in the third integral cohomology of smooth complex projective varieties, respectively the (unramified) Brauer group, that has previously been used by Artin--Mumford \cite{artin-mumford} and Saltman \cite{saltman}.
It has recently found many applications in combination with a cycle-theoretic specialization technique that was initiated by Voisin \cite{voisin} and developed further by  Colliot-Thélène--Pirutka \cite{CTP} and the author \cite{Sch-duke,Sch-JAMS}.

The purpose of this  paper is to give a detailed survey on unramified cohomlogy with a view towards its applications to rationality problems.
We discuss in particular the  interactions of unramified cohomology with cycles and correspondences,  
explain the most important known methods to construct nontrivial unramified cohomology classes in concrete examples and present the aforementioned cycle-theoretic specialization method, which in conjunction with unramified cohomology yields a  powerful obstruction to (stable) rationality. 
In addition, we discuss some the most important functoriality and compatibility results for unramified cohomology.  
This includes in particular the following two results that do not seem to be recorded well in the literature: functoriality with respect to finite pushforwards (see Proposition \ref{prop:functoriality}) and compatibility of the residue map with cup products (see Lemmas \ref{lem:Gysin-cup} and \ref{lem:residue}).

Recent work of Nicaise--Shinder, Kontsevich--Tschinkel, and Nicaise--Ottem \cite{NS,KT,NO} makes the importance of semistable degenerations in the subject apparent.
For this reason, we generalize in this survey Merkurjev's pairing to the case of schemes with normal crossings in Section \ref{sec:snc} and show that it has direct consequences for the aforementioned cycle-theoretic degeneration technique, see Theorem \ref{thm:specialization} below.
 
 The reader is   advised to also consult the excellent survey on unramified cohomology by Colliot-Th\'el\`ene \cite{CT}.
The present survey complements  \cite{CT} with more recent developments, such as Merkurjev's pairing (see Sections \ref{sec:pairing} and \ref{sec:snc}),   decompositions of the diagonal (see Section \ref{sec:diagonal}), the specialization method (see Section \ref{sec:specialization}) and some  important vanishing  results (see Section \ref{sec:vanishing}).
Moreover, some special emphasize on the computational aspect of the subject is given in Section \ref{sec:examples}.
We end this survey with a few open problems in Section \ref{sec:problem} and outline a strategy how unramified cohomology and decompositions of the diagonal may potentially be used to prove the existence of a rationally connected variety that is not unirational, see Proposition \ref{prop:unirational}.

\subsection{Notation and convention}
All schemes are separated and locally Noetherian.  An algebraic scheme is a scheme of finite type over a field.  A variety is an integral algebraic scheme.  If $k$ is an uncountable field, a very general point of a $k$-variety $X$ is a closed point outside a countable union of proper closed (algebraic) subsets.
An alteration of a  variety $X$ over an algebraically closed field $k$ is a proper generically finite morphism $\tau:X'\to X$ such that $X'$ is smooth over $k$.
The existence of alterations was shown by de Jong \cite{deJong}.

\section{Preliminaries from \'etale cohomology}

Let $k$ be a field. 
For a positive integer $m$ that is invertible in $k$, 
we denote by $\mu_m$ the sheaf of $m$-th roots of unity.
For any scheme $X$ over $k$, this is a subsheaf of the multiplicative sheaf $\mathbb G_m$ of invertible functions on $X$ and hence a sheaf in the \'etale topology of $X$.
For an integer $j\geq 1$, we consider the twists $ \mu_m^{\otimes j}:=\mu_m\otimes\dots \otimes \mu_m $ (j-times) and put $\mu_m^{\otimes 0}:=\Z/m$ and $  \mu_m^{\otimes j}:=\Hom(\mu_m^{\otimes - j},\Z/m)$ for $j<0$. 
If $k$ contains all $m$-th roots of unity (e.g.\ if $k$ is algebraically closed), then $\mu_m^{\otimes j}\cong \Z/m$ is a constant sheaf for all $j$.

For a scheme $X$ and a sheaf $F$ in the \'etale topology of $X$, we denote by 
$
H^i(X,F)
$
the $i$-th \'etale cohomology of $F$. 
If $X=\Spec A$ for some ring $A$, then we write
$$
H^i(A,F):=H^i(\Spec A,F).
$$ 

\subsection{Cohomology of fields}
Since $H^1(K,\mathbb G_m)=0$ by Hilbert 90, the long exact sequence associated to the Kummer sequence
\begin{align} \label{seq:Kummer}
0\longrightarrow \mu_m\longrightarrow \mathbb G_m\longrightarrow \mathbb G_m\longrightarrow 0
\end{align} 
shows that for any field $K$ in which $m$ is invertible, there is a canonical isomorphism
$$
H^1(K,\mu_m)\cong K^\ast/(K^\ast)^m,
$$
where $(K^\ast)^m\subset K^\ast$ denotes the subgroup of $m$-th powers in $K^\ast$.
Using this, we denote the class in $H^1(K,\mu_m)$ that is represented by an element $a\in K^\ast$ by $(a)$.
Moreover, for a collection of elements $a_1,\dots ,a_n\in K^\ast$, we denote by the symbol $(a_1,\dots ,a_n)$ the class
\begin{align} \label{eq:symbol}
(a_1,\dots ,a_n):=(a_1)\cup (a_2)\cup \cdots \cup (a_n)\in H^n(K,\mu_m^{\otimes n})
\end{align}
that is given by cup product. 

We will also need the following well-known result (see e.g.\ \cite[II.4.2]{serre}):
for a variety $X$ over an algebraically closed field $k$ in which $m$ is invertible, 
\begin{align} \label{eq:coho-dimension}
H^i(k(X),\mu_m^{\otimes j})=0 \ \ \text{for all $i>\dim X$.}
\end{align}

\subsection{Commutativity with direct limits}
If $X$ is a quasi-projective variety over a field $k$, then for any locally constant torsion \'etale sheaf $F$ and for any point $x\in X$, 
there are canonical isomorphisms
\begin{align} \label{eq:direct-limit}
\lim_{\substack{\longrightarrow \\ x\in U\subset X}}H^i(U,F)\cong H^i\left(\lim_{\substack{\longleftarrow \\ x\in U\subset X}}U,F\right)= H^i(  \mathcal O_{X,x},F) .
\end{align}
Here we used that in the above limit, we may restrict to the directed system of affine open subsets $U\subset X$ containing $x$.
This guarantees that the transition maps  are affine and so the above compatibility between the direct limit of \'etale cohomology groups and the \'etale cohomology of the inverse limit of schemes holds true by \cite[p.\ 119, III.3.17]{milne}, where we use that $F$ is locally constant.
Applying (\ref{eq:direct-limit}) to the generic point of $X$, we find in particular
\begin{align} \label{eq:direct-limit-2}
\lim_{\substack{\longrightarrow \\ \emptyset\neq U\subset X}}H^i(U,F)\cong H^i(k(X),F),
\end{align}
where the limit runs through all non-empty open subsets $U$ of $X$ and $k(X)$ denotes the function field of $X$.


\subsection{Long exact sequence of pairs} \label{subsec:les} 
\begin{proposition}[{\cite[III.1.25]{milne} or \cite[V.6.5.4]{SGA4.2}}]\label{prop:les-pair}
Let $V$ be a  scheme and let $Z\subset V$ be a closed subscheme with complement $U$.
For any \'etale sheaf $F$ on $V$, the pair $(V,Z)$ gives rise to a long exact sequence 
\begin{align} \label{eq:les:pair}
\dots \longrightarrow H^i(V,F)\longrightarrow H^i(U,F|_U)\longrightarrow H^{i+1}_Z(V,F) \longrightarrow H^{i+1}(V,F)\longrightarrow \dots ,
\end{align}
where $H^i_Z(V,F)$ denotes \'etale cohomology of $F$ with support on $Z$.
\end{proposition}
\begin{proof}
Let $i:Z\to V$ and $j:U\to V$ denote the inclusions.
Consider the short exact sequence
$$
0\longrightarrow j_{!}\Z_U\longrightarrow \Z_V\longrightarrow i_\ast \Z_Z\longrightarrow 0
$$
of \'etale sheaves on $V$, where $\Z_V$ denotes the constant sheaf on $V$ with stalk $\Z$, $j_{!}\Z_U$ denotes the sheaf that agrees with $\Z_V$ on $U$ and is 0 outside of $U$ and $i_\ast\Z_Z$ denotes the pushforward of the constant sheaf with stalk $\Z$ on $Z$ to $V$.
For any \'etale sheaf $F$ on $V$, applying $\operatorname{RHom}(-,F)$ to the above short exact sequence, we get a long exact sequence
\begin{align} \label{eq:les:Ext}
\dots \longrightarrow \Ext^i(\Z_V,F)\longrightarrow \Ext^i(j_!\Z_U,F)\longrightarrow \Ext^{i+1}(i_\ast \Z_Z,F)\longrightarrow \Ext^{i+1}(\Z_V,F)\longrightarrow \dots .
\end{align}
The long exact sequence (\ref{eq:les:pair}) follows from this as there are canonical identifications
\begin{align} \label{eq:Ext=H^i}
\Ext^i(j_!\Z_U,F)\cong H^i(U,F),\ \  \Ext^i(\Z_V,F)\cong H^i(V,F),\ \ \Ext^i(i_\ast \Z_Z,F)\cong H^i_Z(V,F) ,
\end{align}
see \cite[proof of III.1.25]{milne} or \cite[V.6.3]{SGA4.2}.
\end{proof}

\subsection{Cup products}
Let $V$ be a scheme and recall that the abelian category $\operatorname{Sh}_{\text{\'et}}(V)$ of \'etale sheaves on $V$ has enough injectives, see e.g.\ \cite[Proposition III.1.1]{milne}.
Let $D=D(\operatorname{Sh}_{\text{\'et}}(V))$ be the corresponding derived category.
For \'etale sheaves $F$ and $G$ on $V$, we then get canonical isomorphisms
$$
 \Ext^j(F,G)\cong \Hom_{D}(F,G[j]),
 $$ 
see e.g.\ \cite[Proposition 2.56]{huybrechts}.
Composition in the derived category thus yields for any \'etale sheaves $F,G$, and $H$ on $V$ a map
$$
\Ext^i(H,F)\otimes \Ext^j(F,  G)\longrightarrow \Ext^{i+j}(H,  G).
$$

There is a canonical map $\Ext^j(\Z_V,G)\to \Ext^j(F,F\otimes G)$ and so  for any \'etale sheaves $F,G$, and $H$ on $V$, the above construction yields a map
$$
\Ext^i(H,F)\otimes \Ext^j(\Z_V,  G)\longrightarrow \Ext^{i+j}(H,  F\otimes G) .
$$
%
%
Via  the identifications in (\ref{eq:Ext=H^i}), any class $\alpha\in H^j(V,G)$  thus yields cup product maps
$$
\cup \alpha : H^i(V,F)\to H^{i+j}(V,F\otimes G),\ \  \cup \alpha|_U : H^i(U,F)\to H^{i+j}(U,F\otimes G),
$$
$$
 \text{and}\ \  
\cup \alpha : H^i_Z(V,F)\to H_Z^{i+j}(V,F\otimes G) .
$$
We will need the following well-known compatibility of the long exact sequence in Proposition \ref{prop:les-pair} with the above cup product maps.

\begin{lemma}\label{lem:les-pair}
Let $V$ be a  scheme and let $Z\subset V$ be a closed subscheme with $U:=V\setminus Z$.
For any \'etale sheaves $F$ and $G$ on $V$ and for any class $\alpha\in H^j(V,G)$,  cup product with $\alpha$ induces a commutative diagram
\begin{align*} 
\xymatrix{
H^i(V,F)\ar[r]\ar[d]^{\cup \alpha} & H^i(U,F)\ar[r]\ar[d]^{\cup \alpha|_U} &H^{i+1}_Z(V,F)\ar[r]\ar[d]^{\cup \alpha} &H^{i+1}(V,F) \ar[d]^{\cup \alpha} \\
H^{i+j}(V,F\otimes G)\ar[r]&H^{i+j}(U,F\otimes G)\ar[r]&H^{i+j+1}_Z(V,F\otimes G)\ar[r]&H^{i+j+1}(V,F\otimes G) 
}
\end{align*}
between the long exact sequences from Proposition \ref{prop:les-pair}.
\end{lemma}
\begin{proof} 
Let $\alpha'\in \Ext^j(F,F\otimes G)$ be the image of $\alpha$ via the canonical map $\Ext^j(\Z_V,G)\to \Ext^j(F,F\otimes G)$.
As explained above, we may think about $\alpha'$ as a morphism $\alpha':F\to F\otimes G[j]$ in the derived category $D=D(\operatorname{Sh}_{\text{\'et}}(V))$.
This morphism yields by \cite[Remark 12.57(ii)]{huybrechts} a morphism between the following distinguished triangles in $D$:
$$
\xymatrix{
\RHom(i_\ast\Z_Z,F)\ar[r]\ar[d]^{\alpha'}& \RHom(\Z_V,F)\ar[r]\ar[d]^{\alpha'}& \RHom(j_!\Z_U,F)\ar[r]^-{+1}\ar[d]^{\alpha'}&\\
\RHom(i_\ast\Z_Z,F\otimes G[j])\ar[r] & \RHom(\Z_V,F\otimes G[j])\ar[r] & \RHom(j_!\Z_U,F\otimes G[j])\ar[r]^-{+1} & 
}
$$
The above morphism between exact triangles yields a morphism between the associated long exact sequences of cohomology groups, which by (\ref{eq:Ext=H^i}) identifies to the commutative diagram in the lemma, as we want.
This concludes the proof.
\end{proof}

%
\subsection{Gysin sequence}

\begin{theorem}[Gysin sequence] \label{thm:gysin}
Let $V$ be a regular Noetherian scheme with a regular closed subscheme $Z\subset V$ of pure codimension $c$ and complement $U:=V\setminus Z$.
Then for any integer $m$ invertible on $V$, there is a long exact sequence
$$
\dots \longrightarrow H^i(V,\mu_m^{\otimes j})\longrightarrow H^i(U,\mu_m^{\otimes {j}})\stackrel{\del}\longrightarrow H^{i+1-2c}(Z,\mu_m^{\otimes j-c}) \stackrel{\iota_\ast} \longrightarrow H^{i+1}(V,\mu_m^{\otimes j})\longrightarrow \dots .
$$
\end{theorem}
\begin{proof}
The theorem is a well-known consequence of the long exact sequence of pairs in Proposition \ref{prop:les-pair} together with Gabber's proof of Grothendieck's purity conjecture \cite{fujiwara,ILO}; we give some details for convenience of the reader.

Let $\underline H^i_Z(V,\mu_{m}^{\otimes j})$ be the \'etale sheaf on $V$, associated to the presheaf $U\mapsto H^i_{Z\times_V U}(U,\mu_{m}^{\otimes j})$, cf.\ \cite[p.\ 85, III.1.6(e)]{milne}.
Then there is a local to global spectral sequence
\begin{align} \label{eq:local-global-ss}
E_2^{p,q}:=H^p(V,\underline H^q_Z(V,\mu_{m}^{\otimes j}))\Longrightarrow H^{p+q}_Z(V,\mu_{m}^{\otimes j}),
\end{align}
see \cite[V.6.4]{SGA4.2}.
Since $Z$ and $V$ are regular and $Z$ is of pure codimension $c$ in $V$,   Grothendiecks purity conjecture, proven by Gabber (see \cite[XVI.3.1.1]{ILO} and \cite[2.1.1]{fujiwara}), implies
$$
\underline H^q_Z(V,\mu_{m}^{\otimes j})\cong \begin{cases}
i_\ast i^\ast\mu_{m}^{\otimes j-c}\ \ &\text{if $q=2c$};\\
0\ \ &\text{else},
\end{cases}
$$
where $i:Z\to V$ denotes the closed immersion.
Hence, the spectral sequence (\ref{eq:local-global-ss}) degenerates at $E_2$ and yields a canonical  isomorphism:
\begin{align}\label{eq:purity}
H^{i-2c}(Z,\mu_m^{\otimes j-c}) \stackrel{\cong}\longrightarrow H^i_Z(V,\mu_m^{\otimes j}).
\end{align} 
Combining this with (\ref{eq:les:pair}), the theorem follows.
\end{proof}

\begin{lemma}\label{lem:Gysin-cup}
Let $V$ be a regular Noetherian scheme with a regular closed subscheme $Z\subset V$ of pure codimension $c$ and complement $U:=V\setminus Z$.
Let $m$ be an integer that is invertible on $V$ and let $\alpha\in H^r(V,\mu_{m}^{\otimes s})$.
Then cup product with $\alpha$ is compatible with the Gysin sequence in Theorem \ref{thm:gysin}, i.e.\ there is a commutative diagram:
\begin{align*} 
\xymatrix{
H^i(V,\mu_{m}^{\otimes j})\ar[r]\ar[d]^{\cup \alpha} & H^i(U,\mu_{m}^{\otimes j})\ar[r]\ar[d]^{\cup \alpha|_U} &H^{i+1-2c}(Z,\mu_{m}^{\otimes j-c})\ar[r]\ar[d]^{\cup \alpha|_Z} &H^{i+1}(V,\mu_{m}^{\otimes j}) \ar[d]^{\cup \alpha} \\
H^{i+r}(V,\mu_{m}^{\otimes j+s})\ar[r]&H^{i+r}(U,\mu_{m}^{\otimes j+s})\ar[r]&H^{i+r+1-2c}(Z,\mu_{m}^{\otimes j+s-c})\ar[r]&H^{i+j+1}(V,\mu_{m}^{\otimes j+s}) 
}
\end{align*}
\end{lemma}
\begin{proof}
It follows from the construction of the isomorphism in (\ref{eq:purity}) via the spectral sequence (\ref{eq:local-global-ss}) that the cup product map
$$
\cup \alpha:H^{i}_Z(V,\mu_{m}^{\otimes j})\cong \Ext^i(i_\ast \Z_Z,\mu_{m}^{\otimes j})\longrightarrow \Ext^{i+r}(i_\ast \Z_Z,\mu_{m}^{\otimes j+s})\cong H^{i+r}_Z(V,\mu_{m}^{\otimes j+s})
$$
identifies via (\ref{eq:purity}) to the map given by cup product with the restriction of $\alpha$ to $Z$:
$$
\cup \alpha|_{Z}:H^{i-2c}(Z,\mu_m^{\otimes j-c}) \longrightarrow H^{i+r-2c}(Z,\mu_m^{\otimes j+s-c}) .
$$
The commutative diagram in question  thus follows directly
from Lemma \ref{lem:les-pair} and Theorem \ref{thm:gysin}.
This concludes the proof.
\end{proof}

\begin{remark}
In algebraic topology, the cohomology group $H^i_Z(V,-)$ with support corresponds to the relative cohomology group $H^i(V,U;-)$ of the pair $(V,U)$ and the long exact sequence (\ref{eq:les:pair}) identifies to the long exact sequence of the pair $(V,U)$.
Moreover, the isomorphism in (\ref{eq:purity}) corresponds to the Thom isomorphism in differential topology, which asserts that if $Z$ is a closed complex submanifold of a complex manifold $V$ of pure codimension $c$ and with complement $U=V\setminus Z$, then 
$$
H^i(V,U;\Z(j))\cong H^{i-2c}(Z,\Z(j-c)),
$$
where $\Z(j):=(2\pi i)^{j} \cdot \Z\subset \C$ denotes the $j$-th Tate twist of $\Z$.
Combining both results one obtains a topological version of the Gysin sequence in Theorem \ref{thm:gysin}.
The compatibility result with cup products stated in Lemma \ref{lem:les-pair} is in this context proven in \cite[VII.8.10]{dold}.
\end{remark}

\section{Residue maps}

Let $A$ be a discrete valuation ring with fraction field $K$ and residue field $\kappa$. 
Applying Theorem \ref{thm:gysin} to $V=\Spec A$, $Z=\Spec \kappa$ and $U=\Spec K$, we get a long exact sequence of the form
\begin{align} \label{eq:les:residue}
\dots \longrightarrow H^i( A,\mu_m^{\otimes j})\longrightarrow H^i(K,\mu_m^{\otimes {j}})\stackrel{\del}\longrightarrow H^{i-1}(\kappa,\mu_m^{\otimes j-1}) \longrightarrow H^{i+1}(A,\mu_m^{\otimes j})\longrightarrow \dots .
\end{align}
One then defines the residue map
\begin{align} \label{def:residue}
\del_A:=-\del:H^i(K,\mu_m^{\otimes {j}})\longrightarrow H^{i-1}(\kappa,\mu_m^{\otimes j-1}),
\end{align}
as the negative of the above  boundary map $\del$.
Here the minus sign is necessary to make our definition compatible with the definition of the residue map in Galois cohomology or Milnor K-theory, cf.\ \cite[\S 3.3]{CT}. 
By definition, the residue map has the property that its kernel coincides with the image of the natural map
$
H^i( A,\mu_m^{\otimes j})\longrightarrow H^i(K,\mu_m^{\otimes {j}}).
$

\begin{lemma} \label{lem:residue}
Let $A$ be a discrete valuation ring with fraction field $K$ and residue field $\kappa$. 
Then for any $\alpha'\in H^{i_1}(A,\mu_m^{\otimes j_1}) $ with image $\alpha\in H^{i_1}(K,\mu_m^{\otimes j_1}) $  and any $\beta \in H^{i_2}(K,\mu_m^{\otimes j_2})$, we have
$$
\del (\beta\cup \alpha)=(\del\beta)\cup \overline \alpha ,
$$
where $\overline \alpha \in H^{i_1}(\kappa,\mu_m^{\otimes j_1})$ denotes the image of $\alpha'$.
\end{lemma}
\begin{proof} This is a direct consequence of Lemma \ref{lem:Gysin-cup}.
\end{proof}

We also need the following well-known lemma, see e.g.\ \cite[Proposition 1.3]{CTO}.

\begin{lemma} \label{lem:del=val}
Let $A$ be a discrete valuation ring with fraction field $K$ and residue field $\kappa$. 
Let $m$ be an integer invertible in $\kappa$ and let $\nu$ denote the valuation on $K$ that corresponds to $A$.
Then the composition
$$
\del:H^1(K,\mu_m)\cong K^\ast/(K^\ast)^m\longrightarrow H^0(\kappa,\Z/m)\cong \Z/m
$$
coincides with the homomorphism that is induced by the valuation $\nu:K^\ast\to \Z$.
\end{lemma}

\begin{remark}
Since the cup product in \'etale cohomology is graded commutative and bilinear, Lemmas \ref{lem:residue} and \ref{lem:del=val}  completely determine the residue map on symbols $(a_1,\dots , a_n)$  as in (\ref{eq:symbol}), cf.\ \cite[Lemma 2.1]{Sch-JAMS}. 
\end{remark}
%

\subsection{Compatibility with pullbacks}
Let $K'/K$ be a field extension and let $m$ be a positive integer that is invertible in $K$.
We can think about this as a morphism $f:\Spec K'\to \Spec K$ of schemes.
Since \'etale cohomology is contravariant with respect to arbitrary morphisms of schemes (see e.g.\ \cite[p.\ 85, III.1.6(c)]{milne}), we thus get pullback maps
$$
f^\ast :H^i(K,\mu_m^{\otimes j})\longrightarrow H^i(K',\mu_m^{\otimes j}).
$$ 
Let $\nu'$ be a discrete valuation on $K'$ with valuation ring $A'$ and residue field $\kappa_{A'}$.
Assume that the restriction $\nu=\nu'|_K$ of $\nu'$ to $K$ is a nontrivial valuation on $K$ with valuation ring $A$ and residue field $\kappa_A$.
This is a subfield of $\kappa_{A'}$ and so we get a natural morphism $g:\Spec \kappa_{A'} \to \Spec \kappa_A$.

Let $\pi$ be a uniformizer of $A$ and let $e:=\nu'(\pi)$ be the valuation of $\pi$ in the valuation ring $A'$, which is a non-negative integer that measures the ramification of the extension $A\subset A'$.
Then we have a commutative diagram
\begin{align} \label{diag:del-pullbacks}
\xymatrix{
H^i(K',\mu_m^{\otimes j}) \ar[rr]^{\del_{A'}} && H^{i-1}(\kappa_{A'},\mu_m^{\otimes j-1})\\
H^i(K,\mu_m^{\otimes j})  \ar[u]^{f^\ast} \ar[rr]^{\del_A} && H^{i-1}(\kappa_A,\mu_m^{\otimes j-1}) \ar[u]^{e\cdot g^\ast} 
};
\end{align}
see e.g.\ \cite[p.\ 143]{CTO}, where the statement is proven after identifying the above \'etale cohomology groups with the corresponding Galois cohomology groups.

\subsection{Compatibility with pushforwards}

Let now $K'/K$ be a finite field extension and let $m$ be a positive integer that is invertible in $K$.
We fix a discrete valuation $\nu$  on $K$ with valuation ring $A$.
Let $A'\subset K'$ be the integral closure of $A$.
By the Krull--Akizuki theorem (see \cite[p.\ 500, Ch. VII, \S 2, no.\ 5, Prop.\ 5]{bourbaki}), $A'$ is a Dedekind domain with finitely many maximal ideals $\mathfrak m_1,\dots, \mathfrak m_r$.
Each localization $A'_l:=A'_{\mathfrak m_l}$ is a discrete valuation ring with fraction field $K'$ and we denote its residue field by $\kappa_{A'_l}$, which (by the aforementioned Krull Akizuki theorem) is a finite extension of $\kappa_A$. 
In particular, the natural morphisms $f:\Spec K'\to \Spec K$ and $g_l:\Spec \kappa_{A'_l} \to \Spec \kappa_A$ are finite, hence proper, and  we get pushforward morphisms
$$
f_\ast :H^i(K',\mu_m^{\otimes j})\longrightarrow H^i(K,\mu_m^{\otimes j}) \ \ \text{and}\ \ (g_l)_\ast :H^i(\kappa_{A'_l},\mu_m^{\otimes j})\longrightarrow H^i(\kappa_A,\mu_m^{\otimes j})  ,
$$ 
for all $i$ and all $l=1,\dots ,r$.
Under the additional assumption that $A'$ is a finite $A$-module, we then get the following compatibility result.

\begin{lemma}
Let $K'/K$ be a finite field extension and let $m$ be a positive integer that is invertible in $K$.
Let $\nu$ be a discrete valuation on $K$ with valuation ring $A\subset K$ and let $A'\subset K'$ be the integral closure of $A$ in $K'$. 
Assume that $A'$ is a finite $A$-module. 
Then in the above notation, the following diagram is commutative:
\begin{align} \label{diag:del-pushforward}
\xymatrix{
H^i(K',\mu_m^{\otimes j}) \ar[d]^{f_\ast}  \ar[rrr]^{\sum_l \del_{A'_l}} &&& \bigoplus_{l=1}^r H^{i-1}(\kappa_{A'_l},\mu_m^{\otimes j-1}) \ar[d]^{\sum_l (g_l)_\ast}  \\
H^i(K,\mu_m^{\otimes j})  \ar[rrr]^{\del_A} &&& H^{i-1}(\kappa_A,\mu_m^{\otimes j-1}) 
}.
\end{align} 
\end{lemma}
\begin{proof}
Let $V':=\Spec A'$ and $U':=\Spec K'$.
Let further $Z'\subset V'$ be  the union of the finitely many closed points $\Spec \kappa_{A'_l}$, $l=1,\dots ,r$.
Similarly, we put $V:=\Spec A$, $U:=\Spec K$ and $Z:=\Spec \kappa_A$.
%
Since $A'$ and $A$ are normal of dimension one, they are both regular and so the  Gysin sequence from Theorem \ref{thm:gysin} holds for  $(V,Z)$ and $(V',Z')$.
We thus have exact sequences
\begin{align}\label{eq:les:del-pushforward-1}
 H^i(A',\mu_m^{\otimes j})\longrightarrow H^i(K',\mu_m^{\otimes {j}})\stackrel{\sum_l \del_{A'_l}} \longrightarrow \bigoplus_l H^{i-1}(\kappa_{A'_l},\mu_m^{\otimes j-1}) \longrightarrow H^{i+1}(A',\mu_m^{\otimes j}) ,
\end{align}
and
\begin{align}\label{eq:les:del-pushforward-2}
 H^i(A,\mu_m^{\otimes j})\longrightarrow H^i(K,\mu_m^{\otimes {j}})\stackrel{\del_A}\longrightarrow H^{i-1}(\kappa_A,\mu_m^{\otimes j-1}) \longrightarrow H^{i+1}(A,\mu_m^{\otimes j}) .
\end{align}
Since $A'$ is a finite extension of $A$, there is a natural pushforward map from each of the terms of (\ref{eq:les:del-pushforward-1}) to the respective terms of (\ref{eq:les:del-pushforward-2}).
To prove the lemma, we need to see that these morphisms yield a morphism between the (exact) complexes in (\ref{eq:les:del-pushforward-1}) and (\ref{eq:les:del-pushforward-2}).

Since the natural morphism $f:V'\to V$ is finite, we have $R^pf_\ast \mu_m^{\otimes j}=0$ for $p\geq 1$, see e.g.\ \cite[VI.2.5]{milne}.
Applying the same reasoning to the base change of $f$ to $U$ and $Z$ (which we denote by the same letter $f$), we obtain natural isomorphisms
$$
H^i(U',\mu_m^{\otimes j})\cong H^i(U,f_\ast \mu_m^{\otimes j}),\ \  H^i(V',\mu_m^{\otimes j})\cong H^i(V,f_\ast \mu_m^{\otimes j})\ \ \text{and}\ \ H^i(Z',\mu_m^{\otimes j})\cong H^i(Z,f_\ast \mu_m^{\otimes j}).
$$
The sequence (\ref{eq:les:del-pushforward-1}) is thus by (\ref{eq:Ext=H^i}) isomorphic to the sequence (\ref{eq:les:Ext}) with $F=f_\ast \mu_m^{\otimes j}$, while (\ref{eq:les:del-pushforward-2}) is isomorphic to (\ref{eq:les:Ext}) with $F= \mu_m^{\otimes j}$.
The norm homomorphism induces a natural morphism of \'etale sheaves $ f_\ast \mu_m^{\otimes j}\to \mu_m^{\otimes j} $ on $V$.
This  induces a morphism from the long exact sequence (\ref{eq:les:Ext}) with $F=f_\ast \mu_m^{\otimes j}$ to that for $F= \mu_m^{\otimes j}$.
Using the above identifications, one checks that this morphism of long exact sequences is nothing but the morphism between (\ref{eq:les:del-pushforward-1})  and (\ref{eq:les:del-pushforward-2}) that is induced by the respective pushforward maps.
This concludes the proof of the lemma. 
\end{proof}

\begin{remark}
The commutativity in (\ref{diag:del-pullbacks}) can be proven similarly as above, by replacing the norm map $ f_\ast \mu_m^{\otimes j}\to \mu_m^{\otimes j} $ by the natural map $ \mu_m^{\otimes j}\to f_\ast \mu_m^{\otimes j} $.
Note however that in this case, the pullback map
$
H^i_Z(V,\mu_m^{\otimes j})\longrightarrow H^i_{Z'}(V',\mu_m^{\otimes j})
$
corresponds via the Gysin isomorphism (\ref{eq:purity}) to the morphism
$$
H^{i-2c}(Z,\mu_m^{\otimes j})\longrightarrow H^{i-2c}(Z',\mu_m^{\otimes j})
$$
that is given by $e$ times the pullback morphism, where $e$ denotes the ramification index of the ring extension $A'/A$.
\end{remark}

%
%
 
\subsection{A consequence of  Bloch--Ogus' theorem}

The following result, known as injectivity and codimension 1 purity property for \'etale cohomology, is a consequence of Bloch--Ogus' proof \cite{BO} of the Gersten conjecture for \'etale cohomology, cf.\ \cite[Theorems 3.8.1 and 3.8.2]{CT}.

\begin{theorem}\label{thm:injectivity+purity}
Let $X$ be a variety over a field $k$ and let $m$ be a positive integer that is invertible in $k$.
Let $x$ be a point in the smooth locus of $X$.
Then the following hold:
\begin{enumerate}[(a)]
\item The natural morphism 
\begin{align} \label{eq:O_X,x}
H^i(\mathcal O_{X,x},\mu_m^{\otimes j})\longrightarrow H^i(k(X),\mu_m^{\otimes j})
\end{align}
is injective.
\item A class  $\alpha\in H^i(k(X),\mu_m^{\otimes j})$ lies in the image of (\ref{eq:O_X,x}) if and only if $\alpha$ has trivial residue along each prime divisor on $X$ that passes through $x$.
\end{enumerate} 
\end{theorem}

\section{Unramified cohomology}

Let $K/k$ be a finitely generated field extension and let $\nu$ be a discrete valuation on $K$.
We say that $\nu$ is a valuation on $K$ over $k$, if $\nu$ is trivial on $k$. 
The valuation ring $A_\nu\subset K$ of $\nu$ is a discrete valuation ring with fraction field $K$ and we denote its residue field by $\kappa_\nu$.
By (\ref{def:residue}), we get a residue map 
$$
\del_\nu:=\del_{A_\nu}:H^i(K,\mu_m^{\otimes j})\longrightarrow H^{i-1}(\kappa_\nu,\mu_m^{\otimes j-1}),
$$
for any positive integer $m$ that is invertible in $k$.
 
\begin{definition} 
Let $K/k$ be a finitely generated field extension.
A geometric valuation $\nu$ on $K$ over $k$ is a discrete valuation on $K$ over $k$ such that the transcendence degree of $\kappa_\nu$ over $k$ is given by 
$$
\operatorname{trdeg}_k(\kappa_\nu)= \operatorname{trdeg}_k(K)-1 .
$$ 
\end{definition}

The following lemma goes back to Zariski, see e.g.\  \cite[Lemma 2.45]{kollar-mori} or \cite[Proposition 1.7]{merkurjev}.

\begin{lemma} \label{lem:geometric-val}
Let $K/k$ be a finitely generated field extension.
A discrete valuation $\nu$ on $K$ over $k$ is geometric if and only if there is a normal $k$-variety $Y$ with $k(Y)\cong K$ such that the valuation $\nu$ corresponds to a prime divisor $E$ on $Y$, i.e.\ for any $\phi\in K^\ast$, $\nu(\phi)=\ord_E(\phi)$, where we think about $\phi$ as a rational function on $Y$.
\end{lemma}

\begin{definition}[\cite{CTO,merkurjev}]\label{def:H_nr}
Let $K/k$ be a finitely generated field extension and let $m$ be a positive integer that is invertible in $k$.
We define the unramified cohomology of $K$ over $k$ with coefficients in $\mu_m^{\otimes j}$ as the subgroup
$$
H^i_{nr}(K/k,\mu_m^{\otimes j})\subset H^i(K,\mu_m^{\otimes j})
$$
that consists of all elements $\alpha \in H^i(K,\mu_m^{\otimes j})$ such that for any geometric valuation $\nu$ on $K$ over $k$, we have $\del_\nu(\alpha)=0$.
\end{definition}


\begin{remark} \label{rem:def-H_nr-CTO}
Unramified cohomology was introduced by Colliot-Thélène and Ojanguren \cite[Definition 1.1.1]{CTO} into the subject.
In their original definition (that is also used in the survey \cite{CT}), they ask $\del_\nu \alpha=0$ for any discrete valuation $\nu$ of $K$ over $k$.
The version in Definition \ref{def:H_nr} where one  restricts to geometric valuations has later been introduced by Merkurjev \cite[\S 2.2]{merkurjev}.
It follows from Proposition \ref{prop:smooth-model} below that both definitions coincide if there is a smooth projective variety $X$ over $k$ with $k(X)=K$.
In particular, both notions coincide for any finitely generated field extension $K/k$ if $k$ is perfect and resolution of singularities is known for varieties over the field $k$ (e.g.\ $\operatorname{char}(k)=0$).
In general it seems however unclear whether both definitions coincide.
We prefer to work with the definition given above, as it has slightly better formal properties (e.g.\ it admits proper pushforwards, see Proposition \ref{prop:functoriality} below).
\end{remark}

\subsection{Stable invariance}
One of the most important properties of unramified cohomology is the fact that it is a stable birational invariant (cf.\ \cite[Proposition 1.2]{CTO} and \cite[Example 2.3]{merkurjev}), that is, it does not change if one passes from a field $K$ to a purely transcendental extension of $K$.

\begin{lemma} \label{lem:stable-invariance}
Let $K/k$ be a finitely generated field extension and let $m$ be a positive integer that is invertible in $k$.
Let $n\in \N$ and let $f:\A^n_K\to \Spec K$ be the structure morphism.
Then the canonical morphism
$$
f^\ast: H^i_{nr}(K/k,\mu_m^{\otimes j}) \stackrel{\sim} \longrightarrow H^i_{nr}(K(\A^n)/k,\mu_m^{\otimes j}) 
$$
is an isomorphism.
\end{lemma}
\begin{proof}
By induction, it suffices to prove the case $n=1$.
By the Faddeev sequence (see \cite[Theorem 6.9.1]{GS}), the following is exact:
$$
0\longrightarrow H^i(K,\mu_m^{\otimes j})\stackrel{f^\ast} \longrightarrow H^i(K(\A^1),\mu_m^{\otimes j})\stackrel{\sum \del_x}\longrightarrow \bigoplus_{x\in \CP^1_K} H^{i-1}(\kappa(x),\mu_m^{\otimes j-1}) ,
$$
where $x$ runs through all closed points of $\CP^1_K$ and $\del_x=\del_{\mathcal O_{\CP^1_K,x}}$.
The zero on the left implies that $f^\ast$ is injective and so it remains to prove surjectivity.
Exactness in the middle of the above sequence shows that any class $\alpha \in H^i_{nr}(K(\A^1)/k,\mu_m^{\otimes j})$ is of the form $f^\ast \beta$ for some $\beta\in H^i(K,\mu_m^{\otimes j})$ and we need to show that if $f^\ast \beta$ is unramified over $k$, $\beta$ must have trivial residue at any geometric valuation of $K$ over $k$.
By Lemma \ref{lem:geometric-val}, any such valuation corresponds to a prime divisor $E$ on some $k$-variety $Y$ with function field $K$.
But then $E\times \CP^1$ is a prime divisor on the $k$-variety $Y\times \CP^1$ with function field $K(\A^1)$ and the fact that $f^\ast \beta$ has trivial residue along this divisor shows by (\ref{diag:del-pullbacks}) that $\beta$ has trivial residue along $E$, as we want. 
(Here we used implicitly the injectivity of the natural map $H^{i-1}(k(E),\mu_m^{\otimes j-1})\to H^{i-1}(k(E)(\A^1),\mu_m^{\otimes j-1})$, which follows from the Faddeev sequence above, applied to $K=k(E)$.)
This concludes the proof of the lemma.
\end{proof}

The above lemma has by (\ref{eq:coho-dimension}) the following immediate consequence. 

\begin{corollary}
Let $X$ be a variety over an algebraically closed field $k$ and let $m$ be a positive integer that is invertible in $k$.
If $X$ is stably rational, then $H^i_{nr}(k(X)/k,\mu_m^{\otimes j})=0$ for all $i>0$.
\end{corollary}

\subsection{Functoriality}
Unramified cohomology has the following functoriality properties.  

\begin{proposition} \label{prop:functoriality}
Let $K'/K/k$ be a finitely generated fields extensions, let $f:\Spec K'\to \Spec K$ be the natural morphism and let $m$ be an integer that is invertible in $k$.
\begin{enumerate}[(a)]
\item Then $f^\ast :H^i(K,\mu_m^{\otimes j})\to H^i(K',\mu_m^{\otimes j})$ induces a pullback map
$$
f^\ast:H^i_{nr}(K/k,\mu_m^{\otimes j})\longrightarrow H^i_{nr}(K'/k,\mu_m^{\otimes j}) .
$$
\item If $f$ is finite, then   $f_\ast :H^i(K',\mu_m^{\otimes j})\to H^i(K,\mu_m^{\otimes j})$ induces a pushforward map
$$
f_\ast:H^i_{nr}(K'/k,\mu_m^{\otimes j})\longrightarrow H^i_{nr}(K/k,\mu_m^{\otimes j}) 
$$
with $f_\ast \circ f^\ast=\deg(f)\cdot \id$.
\end{enumerate}
\end{proposition}
\begin{proof}
Item (a) follows directly from  (\ref{diag:del-pullbacks}).

In item (b) the equality $f_\ast \circ f^\ast=\deg(f)\cdot \id$ holds already on the level of \'etale cohomology
 and so it suffices to show that $f_\ast$ is well-defined on the subgroup of unramified classes.
 To this end, let $\alpha\in H^i_{nr}(K'/k,\mu_m^{\otimes j})$ and let $\nu$ be a geometric valuation on $K$ over $k$.
 We then need to show that $\del_\nu f_\ast \alpha=0$.
Since $\nu$ is geometric, we can construct a normal projective variety $Y$ with $k(Y)\cong K$ such that $\nu$ corresponds to the valuation associated to a codimension 1 point $y\in Y^{(1)}$.
Since $K'/K$ is a finite field extension, we can further construct a normal projective variety $Y'$ with a surjective morphism $f:Y'\to Y$ such that $k(Y')\cong K'$ and  $f^\ast:k(Y)\to k(Y')$ corresponds to the given inclusion $K\subset K'$.
Since $f$ is generically finite and surjective, and $y\in Y$ is a codimension 1 point, the reduced preimage $(f^{-1}(y))^{\red}$ is given by finitely many codimension 1 points $y'_1,\dots ,y'_r$ of $Y'$.
We claim that the integral closure $A'\subset K'$ of $A:=\mathcal O_{Y,y}$ in $K'$ is nothing but the local ring of $Y'$ at the finitely many codimension 1 points $y'_1,\dots ,y'_r$ of $Y'$.
Indeed, $\Spec A'\to \Spec A$ can be constructed by first taking the fibre product of the proper morphism $f:Y'\to Y$ with the inclusion $\Spec \mathcal O_{Y,y}\to Y$ and precomposing this with the normalization map   $\Spec A'\to Y'\times_Y \Spec \mathcal O_{Y,y}$.
This description also shows that $\Spec A'\to \Spec A$ is proper, hence finite as it is clearly quasi-finite.
Hence, $A'$ is a finite ring extension of $A$.
Moreover, the localizations of $A'$ at its finitely many maximal ideals are exactly the local rings $\mathcal O_{Y',y'_l}$, $l=1,\dots ,r$.
Since $\alpha \in H^i_{nr}(K'/k,\mu_m^{\otimes j})$ is unramified, we know that
$$
\del_{ \mathcal O_{Y',y'_l}}\alpha=0
$$ 
for all $l=1,\dots ,r$
and so the commutative diagram (\ref{diag:del-pushforward}) shows that  $\del_\nu f_\ast \alpha=0$, as we want.
This completes the proof.
\end{proof}

\subsection{Restriction to scheme points and pullbacks for morphisms between smooth projective varieties}
Theorem \ref{thm:injectivity+purity} has the following important consequences, which show that unramified classes on smooth projective varieties can be restricted to any scheme point.

\begin{proposition} \label{prop:restriction}
Let $X$ be a smooth variety over a field $k$ and let $m$ be a positive integer that is invertible in $k$.
Let $\alpha\in H^i_{nr}(k(X)/k,\mu_m^{\otimes j})$.
\begin{enumerate}[(a)]
\item Then  for any $x\in X$, there is a well-defined restriction 
$$
\alpha|_x\in H^i(\kappa(x),\mu_m^{\otimes j}).
$$
\item If $X$ is also proper over $k$, then $\alpha|_x\in H^i(\kappa(x)/k,\mu_m^{\otimes j})$ is unramified over $k$.
\end{enumerate}
\end{proposition}
\begin{proof}
We begin with the proof of (a).
By part (b) in Theorem \ref{thm:injectivity+purity}, we know that $\alpha$ admits a lift $\tilde \alpha\in H^i(\mathcal O_{X,x},\mu_m^{\otimes j})$ and this lift is unique by part (a) of Theorem \ref{thm:injectivity+purity}.
The restriction $\alpha|_{x}$ may thus be defined as image of $\tilde \alpha$ via the natural morphism
$$
H^i(\mathcal O_{X,x},\mu_m^{\otimes j}) \longrightarrow  H^i(\kappa(x),\mu_m^{\otimes j}).
$$
This concludes the proof of (a).

To prove (b), let $Z$ be a normal variety over $k$ with $k(Z)\cong \kappa(x)$ and let $z\in Z^{(1)}$ be a codimension 1 point of $Z$.
We then have to prove that $\del_z (\alpha|_x)=0$, or, equivalently, that $\alpha|_x$ lies in the image of the natural map
\begin{align} \label{eq:lem:restriciton-2}
 H^i(\mathcal O_{Z,z},\mu_m^{\otimes j})\longrightarrow  H^i(\kappa(x),\mu_m^{\otimes j}) .
\end{align}
Since $X$ is proper, we may up to shrinking $Z$ around $z$ assume that the isomorphism $k(Z)\cong \kappa(x)$  of fields is induced by a morphism of schemes $\iota:Z\to X$ such that the generic point of $Z$ maps to $x\in X$.
Since $\alpha\in H^i_{nr}(k(X)/k,\mu_m^{\otimes j})$ is unramified over $k$, Theorem \ref{thm:injectivity+purity} implies that $\alpha$ lies in 
\begin{align} \label{eq:lem:restriction}
H^i(\mathcal O_{X,\iota(z)},\mu_m^{\otimes j}) \subset H^i(k(X),\mu_m^{\otimes j}) .
\end{align}
By (\ref{eq:direct-limit}), this means that there is an open neighbourhood $U\subset X$ of $\iota(z)$ and a class $\tilde \alpha\in H^i(U,\mu_m^{\otimes j})$ that restricts to $\alpha$ on the generic point $\Spec k(X)$.
Since $\iota(z)$ lies in the closure of the point $x\in X$, it follows that $U$ contains $x$ and so $\tilde \alpha$ has an image in 
$$
H^i(\mathcal O_{X,x},\mu_m^{\otimes j})\subset H^i(k(X),\mu_m^{\otimes j})
$$ 
which must be $\alpha$ by (\ref{eq:lem:restriction}).
In particular, the restriction $\alpha|_x\in H^i(\kappa(x),\mu_m^{\otimes j})$ that we defined above coincides with the image of $\tilde \alpha$ via the natural map
$$
H^i(U,\mu_m^{\otimes j})\longrightarrow  H^i(\kappa(x),\mu_m^{\otimes j}).
$$
Since $\iota(z)\in U$, this shows that $\alpha|_x$ lies in the image of (\ref{eq:lem:restriciton-2}), as we want.
This concludes the proof of the proposition.
\end{proof}

\begin{corollary} \label{cor:pullback}
Let $f:X\to Y$ be a morphism between smooth proper varieties over a field $k$ and let  $m$ be a positive integer that is invertible in $k$.
Then
there is a well-defined pullback map
$$
f^\ast : H^i_{nr}(k(Y)/k,\mu_m^{\otimes j})\longrightarrow  H^i_{nr}(k(X)/k,\mu_m^{\otimes j})
$$
which is given by restricting a given unramified class $\alpha\in H^i_{nr}(k(Y)/k,\mu_m^{\otimes j})$ to the generic point of the image of $f$ and pulling that back to $k(X)$.
\end{corollary}
\begin{proof}
By Proposition \ref{prop:restriction}, the restriction of $\alpha$ to the generic point of the image of $f$ is unramified over $k$ and so is its pullback to $k(X)$ by Proposition \ref{prop:functoriality}.
This proves the corollary.
\end{proof}

\subsection{It is enough to check residues on a smooth proper model}
The codimension 1 purity property (see item (b) in Theorem \ref{thm:injectivity+purity}) implies the following, cf.\ \cite[Theorem 4.1.1]{CT}.

\begin{proposition} \label{prop:smooth-model}
Let $X$ be a smooth proper variety over a field $k$ and let $m$ be a positive integer that is invertible in $k$.
Then a class $\alpha\in H^i(k(X),\mu_m^{\otimes j})$ is unramified over $k$ if and only if $\alpha$ has trivial residue along any prime divisor on $X$. 
\end{proposition}
\begin{proof}
One direction is trivial.
For the converse, assume that $\alpha\in H^i(k(X),\mu_m^{\otimes j})$ has trivial residue along any prime divisor on $X$.
We then need to show that for any normal variety $Y$ over $k$ that is birational to $X$ and for any codimension 1 point $y\in Y$, we have $\del_y\alpha=0$.
Equivalently, we need to see that $\alpha$ lies in the image of the natural map
$$
H^i(\mathcal O_{Y,y},\mu_m^{\otimes j})\longrightarrow H^i(k(X),\mu_m^{\otimes j}).
$$
Since $Y$ is normal, the rational map $\phi:Y\dashrightarrow X$ is defined in codimension 1 and so up to shrinking $Y$ around the codimension 1 point $y$, we may assume that $\phi:Y\to X$ is a morphism.
Let $x=\phi(y)\in X$ be the image of $y$.
Then we get a commutative diagram
$$
\xymatrix{
H^i(\mathcal O_{X,x},\mu_m^{\otimes j}) \ar[rd]\ar[dd] & \\
& H^i(k(X),\mu_m^{\otimes j}) \\
H^i(\mathcal O_{Y,y},\mu_m^{\otimes j}) \ar[ru] & 
}
$$
and so our claim follows from part (b) in Theorem \ref{thm:injectivity+purity}.
\end{proof}

\subsection{Comparison with usual cohomology}
Let $X$ be a smooth proper variety over a field $k$ and let $m$ be invertible in $k$.
It follows immediately from Proposition \ref{prop:smooth-model} that the image of the natural map 
$$
H^i(X,\mu_m^{\otimes j})\longrightarrow H^i (k(X),\mu_m^{\otimes j})
$$
lies in the subgroup of unramified classes and so we get a well-defined map
\begin{align} \label{eq:H^i(X)toH_nr}
H^i(X,\mu_m^{\otimes j})\longrightarrow H^i_{nr}(k(X)/k,\mu_m^{\otimes j}) .
\end{align}
It is not hard to show that this is an isomorphism for $i=1$ (see \cite[Proposition 4.2.1]{CT}) and surjective for $i=2$ and $j=1$.
However, starting from $i=3$, this map is in general neither injective nor surjective.

The kernel of (\ref{eq:H^i(X)toH_nr}) consists of all cohomology classes $\alpha\in H^i(X,\mu_m^{\otimes j})$ that vanish on some non-empty Zariski open subset of $X$.
For instance, if $i=2j$ is even and $\alpha=c_j(E)$ is the Chern class of a vector bundle $E$, then $\alpha$ lies in the kernel of (\ref{eq:H^i(X)toH_nr}) because any vector bundle is generically trivial.
For $i=2$ and $j=1$, this observation can be used to prove the following,  see \cite[Proposition 4.2.3]{CT}.

\begin{proposition}
Let $X$ be a smooth projective variety over a field $k$ and let $m$ be invertible in $k$.
For $i=2$ and $j=1$, the natural map (\ref{eq:H^i(X)toH_nr}) is surjective and its kernel is given by the image of $c_1:\Pic X\to H^2(X,\mu_m)$.
In particular, there is a natural isomorphism
$$
H^2_{nr} (k(X)/k,\mu_m)\cong \frac{H^2(X,\mu_m )}{\im (c_1:\Pic X\to H^2(X,\mu_m))} .
$$
\end{proposition}

It follows from the Kummer sequence (\ref{seq:Kummer}) that the right hand side in the above isomorphism is isomorphic to the $m$-torsion of the Brauer group  $\operatorname{Br}(X)=H^2(X,\mathbb G_m)$ of $X$.

\begin{remark}
For $k=\C$, Colliot-Thélène--Voisin \cite{CTV} found a relation between the third unramified cohomology of a smooth projective variety and the failure of the integral Hodge conjecture for codimension two cycles on $X$.
A  generalization of this result to cycles of arbitrary codimension was recently given in \cite{Sch-refined}.
\end{remark}

\section{Merkurjev's pairing} \label{sec:pairing}

Let $X$ be a smooth proper variety over a field $K$ and let $m$ be invertible in $K$.
(Here it is important to allow $K$ to be non-algebraically closed, e.g.\ the function field of a variety over a smaller field $k$.)
Let $Z_0(X)$ denote the group of 0-cycles on $X$, i.e.\ the free abelian group generated by the closed points of $X$. 

For a closed point $z\in X $, we denote by $f_z:\Spec \kappa(z)\to \Spec K$ the structure morphism.
Following Merkurjev \cite[\S 2.4]{merkurjev}, we then define for any unramified class $\alpha\in H^i_{nr}(K(X)/K,\mu_m^{\otimes j})$ a class
$$
\langle z,\alpha \rangle:=(f_z)_\ast (\alpha|_{z})\in H^i(K,\mu_m^{\otimes j}) ,
$$
where $\alpha|_{z}\in H^i(\kappa(z),\mu_m^{\otimes j})$ denotes the restriction from Proposition \ref{prop:restriction}.
We may extend this definition to arbitrary 0-cycles $z\in Z_0(X)$ linearly and so we obtain a 
bilinear pairing
\begin{align} \label{def:pairing}
Z_0(X)\times H^i_{nr}(K(X)/K,\mu_m^{\otimes j})\longrightarrow H^i(K,\mu_m^{\otimes j}),\ \ (z,\alpha)\mapsto \langle z,\alpha \rangle .
\end{align} 
 
The main result about this pairing is the following proposition due to Merkurjev, which shows that the pairing descends to the level of Chow groups, cf.\ \cite[\S 2.4]{merkurjev}.

\begin{proposition} \label{prop:pairing}
Let $K$ be a field and let $m$ be an integer that is invertible in $K$.
Let $g:C\to X$ be a non-constant morphism between smooth proper $K$-varieties, where $C$ is a curve. 
Then for any  $\alpha\in H^i_{nr}(K(X)/K,\mu_m^{\otimes j})$ and any non-zero rational function $\phi\in K(C)$, we have
$$
\langle g_\ast \operatorname{div}(\phi),\alpha \rangle =0 ,
$$
where $\operatorname{div}(\phi)\in \Div(C)$ denotes the divisor of zeros and poles of $\phi$.
\end{proposition}

Before we prove Proposition \ref{prop:pairing} in Section \ref{subsubsec:proof:pairing}, let us explain some of its applications.

\subsection{Applications of  Proposition \ref{prop:pairing}} \label{subsubsec:application:pairing} 

\begin{corollary} \label{cor:pairing}
Let $X$ be a smooth proper variety over a field $K$.
Then  (\ref{def:pairing}) descends to a bilinear pairing 
\begin{align} \label{def:pairing:Chow}
\CH_0(X )\times H^i_{nr}(K(X)/K,\mu_m^{\otimes j})\longrightarrow H^i(K,\mu_m^{\otimes j}),\ \ (z,\alpha)\mapsto \langle z,\alpha \rangle .
\end{align}
\end{corollary}
\begin{proof}
This is an immediate consequence of Proposition \ref{prop:pairing}.
\end{proof}

The next result is originally due to Karpenko and Merkurjev, see \cite[RC-I]{karpenko-merkurjev}.

\begin{corollary} \label{cor:pairing-Gamma}
Let $X$ and $Y$ be smooth proper varieties over a field $k$ and let $m$ be an integer that is invertible in $k$.
Then there is a bilinear pairing
$$
\CH_{\dim X}(X\times Y)\times H^i_{nr}(k(Y)/k,\mu_m^{\otimes j})\longrightarrow H^i_{nr}(k(X)/k,\mu_m^{\otimes j}),\ \ (\Gamma,\alpha)\mapsto \Gamma^\ast \alpha,
$$
which via linearity is defined as follows: if $\Gamma\subset X\times Y$ is integral and does not dominate the first factor, then $\Gamma^\ast \alpha:=0$; otherwise, the first projection induces a finite morphism $p:\Spec \kappa(\gamma)\to k(X)$, where  $\gamma$ denotes the generic point of $\Gamma$, and we put
$$
\Gamma^\ast \alpha:= p_\ast ( (\pr_2^\ast \alpha)|_{\gamma}).
$$ 
\end{corollary}
\begin{proof}
Assume that $\Gamma\subset X\times Y$ is a subvariety of dimension $\dim X$ that dominates $X$ via the first projection and let $\gamma$ be the generic point of $\Gamma$.
By Proposition \ref{prop:restriction}, $ (\pr_2^\ast \alpha)|_{\gamma}\in H^i_{nr}(\kappa(\gamma)/k,\mu_m^{\otimes j})$ is unramified over $k$ and so 
$$
  p_\ast ( (\pr_2^\ast \alpha)|_{\gamma}) \in H^i_{nr}(k(X)/k,\mu_m^{\otimes j})
$$ 
is unramified over $k$ by Proposition \ref{prop:functoriality}.
Hence, our definition of $\Gamma^\ast \alpha$ is well-defined and we get a bilinear pairing
\begin{align} \label{def:pairing-Gamma}
Z_{n}(X\times Y)\times H^i_{nr}(k(Y)/k,\mu_m^{\otimes j})\longrightarrow H^i_{nr}(k(X)/k,\mu_m^{\otimes j}),\ \ (\Gamma,\alpha)\mapsto \Gamma^\ast \alpha .
\end{align}
It remains to see that this pairing descends to the level of Chow groups.
For this, let $K:=k(X)$ denote the function field of $X$.
Then there are natural group homomorphisms
$$
Z_{\dim X}(X\times Y)\longrightarrow Z_0(Y_K)\ \ \text{and}\ \ H^i_{nr}(k(Y)/k,\mu_m^{\otimes j})\longrightarrow H^i_{nr}(K(Y)/K,\mu_m^{\otimes j}) .
$$
Since $K=k(X)$ and $H^i_{nr}(K/k,\mu_m^{\otimes j})\subset H^i(K,\mu_m^{\otimes j})$, this induces a diagram
$$
\xymatrix{
Z_{\dim X}(X\times Y)\times H^i_{nr}(k(Y)/k,\mu_m^{\otimes j})\ar[rd]\ar[dd] & \\
& H^i(K,\mu_m^{\otimes j})\\
Z_0(Y_K)\times H^i_{nr}(K(Y)/K,\mu_m^{\otimes j})\ar[ru] &  
}
$$
which is commutative by the definition of the pairings in (\ref{def:pairing}) and (\ref{def:pairing-Gamma}).
Since the natural map $Z_{\dim X}(X\times Y)\longrightarrow Z_0(Y_K)$ descends to a map $\CH_{\dim X}(X\times Y)\to \CH_0(Y_K)$, we deduce from Corollary \ref{cor:pairing} that the pairing (\ref{def:pairing-Gamma}) satisfies $\Gamma^\ast \alpha=0$ whenever $\Gamma$ is a cycle that is rationally equivalent to 0.
This concludes the proof of the corollary.
\end{proof}

\subsection{Proof of Proposition \ref{prop:pairing}} \label{subsubsec:proof:pairing} 
For the proof of Proposition \ref{prop:pairing}, we will need the following compatibility result for the pairing defined in (\ref{def:pairing}).

\begin{lemma} \label{lem:pairing}
Let  $g:X\to Y$ be a morphism between smooth proper $K$-varieties.
The pairing (\ref{def:pairing}) has the following properties.
\begin{enumerate}[(i)]
\item 
For any $\alpha\in H^i_{nr}(K(Y)/K,\mu_n^{\otimes j})$ and any $z\in Z_0(X)$, we have
$$
\langle g_\ast z,\alpha \rangle= \langle z,g^\ast \alpha\rangle ,
$$
where $g^\ast \alpha\in H^i_{nr}(K(X)/K,\mu_m^{\otimes j})$ is defined by Corollary \ref{cor:pullback}.
\item If $X$ and $Y$ are curves and $g$ is finite and surjective, then for any $\beta\in H^i_{nr}(K(X)/K,\mu_n^{\otimes j})$ and any $w\in Z_0(Y)$, we have
$$
\langle g^\ast w,\beta \rangle=\langle w,g_\ast \beta\rangle,
$$
where  $g_\ast \beta \in H^i_{nr}(K(Y)/K,\mu_m^{\otimes j})$ is defined by Proposition \ref{prop:functoriality} and $g^\ast w$ denotes the flat pullback of cycles.
\end{enumerate}
\end{lemma}
\begin{proof}
By linearity, it suffices to prove (i) in the case where $z$ is a closed point of $X$.
Let $f_z:\Spec \kappa(z)\to \Spec K$ be the structure morphism.
Then we have
$$
 \langle z,g^\ast \alpha\rangle=(f_z)_\ast (g^\ast \alpha)|_z .
$$
If $y=g(z)$ denotes the image of $z$ in $Y$, with structure morphism $f_y:\Spec \kappa(y)\to \Spec K$, then $g$ induces a morphism $g_z:\Spec \kappa(z)\to \Spec \kappa(y)$ with $f_z=f_y\circ g_z$ and so we find
$$
 \langle z,g^\ast \alpha\rangle
 =(f_z)_\ast (g^\ast \alpha)|_z
 =(f_y\circ g_z)_\ast (g_z^\ast (\alpha|_y))
 =\deg(g_z)\cdot (f_y)_\ast (\alpha|_y) ,
$$
where we used $(g_z)_\ast\circ (g_z)^\ast=\deg(g_z) $.
On the other hand, $g_\ast z=\deg(g_z)\cdot y$ and so
$$
\langle g_\ast z,\alpha \rangle=\langle \deg(g_z)\cdot y,\alpha \rangle=\deg(g_z)\cdot (f_y)_\ast (\alpha|_y).
$$
This proves item (i) of the lemma.

To prove item (ii), it suffices as before to deal with the case where $w$ is a closed point of $Y$.
Since $g$ is finite and surjective and $X$ and $Y$ are both smooth and proper curves, $g$ is flat and so the pullback $g^\ast w$ is defined on the level of cycles.
Explicitly, it is given by
$$
g^\ast w=\sum_{l=1}^r a_r\cdot z_l
$$
where $z_1,\dots ,z_r$ denote the closed points of $X$ that lie above $w$ and where $a_r$ denotes the ramification indices of $\mathcal O_{Y,w}\subset \mathcal O_{X,z_i}$ (recall that $X$ and $Y$ are smooth proper curves by the assumption in (ii)).
Then we have
$$
\langle g^\ast w,\beta \rangle=\sum_{l=1}^r a_l\cdot \langle z_l,\alpha \rangle=\sum_{l=1}^r a_l\cdot (f_{z_l})_\ast (\alpha|_{z_l}),
$$
where $f_{z_l}:\Spec \kappa(z_l)\to \Spec K$ denotes the structure morphism.
On the other hand, if $f_w:\Spec \kappa(w)\to \Spec K$ denotes the structure morphism and $g_{z_l}:\Spec \kappa(z_l)\to \Spec \kappa(w)$ denotes the natural morphism induced by $g$, then
$$
\langle w,g_\ast \beta\rangle=(f_w)_\ast(g_\ast \beta)|_{w} .
$$
To simplify this further, let $\pi\in \mathcal O_{Y,w}$ be a parameter.  
The rational function $\pi$ yields a class $(\pi)\in H^1(k(Y),\mu_m)\cong k(Y)^\ast/(k(Y)^\ast)^m$ and we have by Lemmas \ref{lem:residue} and \ref{lem:del=val} the following well-known formula
$$
(g_\ast \beta)|_{w}=\del_w (g_\ast \beta\cup (\pi)).
$$
By the projection formula,
$$
g_\ast \beta\cup (\pi)=g_\ast(\beta \cup (g^\ast \pi)).
$$
Using the compatibility (\ref{diag:del-pushforward}), we thus get
$$
\del_w (g_\ast \beta\cup (\pi))=\del_w g_\ast(\beta \cup (g^\ast \pi))=\sum_{l=1}^r (g_{z_l})_\ast \del_{z_l}(\beta \cup (g^\ast \pi)).
$$
Since $\beta$ is unramified and $g^\ast \pi\in  \mathcal O_{X,z_l}$ coincides up to a unit with the $a_l$-th power of a parameter of $  \mathcal O_{X,z_l}$, we deduce from Lemmas \ref{lem:residue} and \ref{lem:del=val} that
$$
\del_{z_l}(\beta \cup (g^\ast \pi))=a_l\cdot \beta|_{z_l}.
$$
Putting everything together, this yields
$$
\langle w,g_\ast \beta\rangle=(f_w)_\ast \del_w(g_\ast \beta\cup (\pi))=(f_w)_\ast(\sum_{l=1}^r a_l\cdot  (g_{z_l})_\ast  \beta|_{z_l} ) =
\sum_{l=1}^r a_l\cdot (f_{z_l})_\ast (\alpha|_{z_l}) ,
$$
because $f_{z_l}=f_w\circ g_{z_l}$. 
Hence, $\langle g^\ast w,\beta \rangle=\langle w,g_\ast \beta\rangle$, which concludes the proof of the lemma.
\end{proof}

\begin{proof}[Proof of Proposition \ref{prop:pairing}]
There is a finite morphism $\varphi:C\to \CP^1_K$ with  
$$
\operatorname{div}(\phi)=\varphi^\ast (0-\infty) .
$$ 
By Lemma \ref{lem:pairing}, we thus find
\begin{align*}
\langle g_\ast \operatorname{div}(\phi),\alpha \rangle=\langle g_\ast \varphi^\ast(0-\infty) ,\alpha \rangle&=\langle  \varphi^\ast(0-\infty),g^\ast \alpha \rangle\\
&= \langle  0-\infty,\varphi_\ast g^\ast \alpha \rangle \\
&= \langle  0,\varphi_\ast g^\ast \alpha \rangle -\langle  \infty,\varphi_\ast g^\ast \alpha \rangle  .
\end{align*}
We claim that this last difference vanishes.
To see this, note that $$
\varphi_\ast g^\ast \alpha\in H^i_{nr}(K(\CP^1)/K,\mu_m^{\otimes j}) 
$$
is unramified over $K$ by  Proposition \ref{prop:functoriality} and Corollary \ref{cor:pullback}.
Hence, Lemma \ref{lem:stable-invariance} implies that there is a class $\alpha' \in H^i(K,\mu_m^{\otimes j}) $ with
$$
\varphi_\ast g^\ast \alpha=f^\ast \alpha',
$$
where $f:\CP^1_K\to \Spec K$ denotes the structure morphism.
Using this, we find by the above calculation that
$$
\langle g_\ast \operatorname{div}(\phi),\alpha \rangle= \langle  0,f^\ast \alpha' \rangle -\langle  \infty,f^\ast \alpha' \rangle = \langle  f_\ast 0, \alpha' \rangle -\langle f_\ast \infty, \alpha' \rangle =0,
$$
where we used item (i) of Lemma \ref{lem:pairing} in the second equality and $f_\ast 0=f_\ast \infty \in Z_0(\Spec K)$ in the last equality. 
 This proves the proposition.
\end{proof}

\section{Generalization to schemes with normal crossings} \label{sec:snc}

Let $k$  be a field.
Let $X$ be a pure-dimensional algebraic scheme over $k$.
If $X_i$ with $i\in I$ denote the irreducible components of $X$, then for any non-empty subset $J\subset I$, we define
$$
X_J:=\bigcap_{l\in J} X_l.
$$

\begin{definition}
A pure-dimensional algebraic scheme $X$ over a field $k$ with irreducible components $X_i$ with $i\in I$ is called snc (simple normal crossings)  scheme if for each non-empty subset $J\subset I$, the subscheme $X_J\subset X$ is either empty or smooth of codimension $|J|$, the cardinality of $J$. 
\end{definition}

\begin{definition} \label{def:H_nr-snc}
Let $X$ be a proper snc scheme over a field $k$ and let $m$ be a positive integer invertible in $k$.
We then define the unramified cohomology of $X$ with values in $\mu_m^{\otimes j}$ as the subgroup
$$
H^i_{nr}(X/k,\mu_m^{\otimes j})\subset \bigoplus_{l\in I} H^i_{nr}(k(X_l)/k,\mu_m^{\otimes j}) 
$$
that consists of all collections $\alpha=(\alpha_l)_{l\in I}$ of unramified classes $\alpha_l\in H^i(k(X_l)/k,\mu_m^{\otimes j}) $, such that for all $l,l'\in I$
$$
\alpha_{l}|_{X_l\cap X_{l'}}= \alpha_{l'}|_{X_l\cap X_{l'}},
$$
by which we mean that $\alpha_l$ and $\alpha_{l'}$ have the same restriction to each component of $X_l\cap X_{l'}$ via the restriction maps from Proposition \ref{prop:restriction}.
\end{definition}

If $X$ is a smooth proper variety over $k$, then
$$
H^i_{nr}(X/k,\mu_m^{\otimes j})=H^i_{nr}(k(X)/k,\mu_m^{\otimes j}).
$$
 
For any nonempty subset $J\subset I$, we get a well-defined restriction map
$$
H^i_{nr}(X/k,\mu_m^{\otimes j})\longrightarrow H^i_{nr}(X_J/k,\mu_m^{\otimes j}),\ \ \alpha\mapsto \alpha|_{X_J}
$$
which is defined by picking any index $l\in J$ and defining $ \alpha|_{X_J}$ on each component of $X_J$ as restriction of $\alpha_l$.
This is well-defined (i.e.\ does not depend on the choice of $l$) by the compatibility condition in Definition \ref{def:H_nr-snc}.

\subsection{A pairing on the level of 0-cycles}
Let $X$ be a proper snc scheme over a field $K$ with irreducible components $X_l$, $l\in I$.
Then there is a bilinear pairing 
$$
Z_0(X)\times H^i_{nr}(X/K,\mu_m^{\otimes j})\longrightarrow H^i(K,\mu_m^{\otimes j}),
$$
defined by
\begin{align} \label{def:pairing-snc}
 (z,\alpha)\mapsto \langle z,\alpha \rangle:=\sum_{\emptyset \neq J\subset I} (-1)^{|J|-1} \langle z|_{X_J},\alpha|_{X_J} \rangle ,
\end{align}
where we note that $X_J$ is smooth and where $z|_{X_J}$ denotes the 0-cycle given by intersecting the 0-cycle $z\in Z_0(X)$ with $X_J\subset X$ in the naive sense, i.e.\ we simply keep that part of $z$ that lies on $X_J$ (this is an operation on the level of cycles that does not pass to the level of Chow groups).\footnote{The formula in (\ref{def:pairing-snc}) is motivated by the specialization formula of Nicaise and Shinder in \cite{NS} which itself is motivated by formulas in motivic integration.}

\begin{lemma} \label{lem:pairing-snc}
Let $X$ be a proper snc scheme over a field $K$ and let $m$ be a positive integer that is invertible in $K$.
If $z$ is a closed point of $X$ with structure morphism $f_z:\Spec \kappa(z)\to \Spec K$ and $\alpha\in H^i_{nr}(X/K,\mu_m^{\otimes j})$, then for any component $X_l$ with $z\in X_l$, we have
$$
\langle z,\alpha \rangle= (f_z)_\ast (\alpha_{l}|_{z}) .
$$
\end{lemma}
\begin{proof}
%
%
By definition
$$
\langle z,\alpha \rangle= \sum_{\substack{\emptyset \neq J\subset I \\ z\in X_{J}}} (-1)^{|J|-1}  (f_z)_\ast ((\alpha|_{X_J})|_{z}) .
$$ 
Up to replacing $X$ by the union of those components that contain $z$, we may assume that $z\in X_l$ for all $l\in I$.
We then find
$$
\langle z,\alpha \rangle= \sum_{ \emptyset \neq J\subset I } (-1)^{|J|-1}  (f_z)_\ast ((\alpha|_{X_J})|_{z}) .
$$
The compatibility condition in Definition \ref{def:H_nr-snc} ensures that
$$
(f_z)_\ast ((\alpha|_{X_J})|_{z})=(f_z)_\ast ((\alpha|_{X_{J'}})|_{z}) .
$$
for all non-empty $J,J'\subset I$.
Hence, for any $l\in I$, we have
\begin{align*}
\langle z,\alpha \rangle &= \sum_{ \emptyset \neq J\subset I } (-1)^{|J|-1}  (f_z)_\ast (\alpha|_{X_J}|_{z}) \\
&= - \sum_{ \emptyset \neq J\subset I } (-1)^{|J|}  (f_z)_\ast  (\alpha_{l}|_{z})\\
&= -\sum_{r=1}^{|I|} \binom{|I|}{r}(-1)^r \cdot  (f_z)_\ast  (\alpha_{l}|_{z})\\
&= (f_z)_\ast  (\alpha_{l}|_{z}) ,
\end{align*}
where we used in the last equality that $\sum_{r=1}^{|I|} \binom{|I|}{r}(-1)^r=-1$.
This proves the lemma.
\end{proof}

\begin{proposition} \label{prop:pairing-snc}
Let $K$ be a field and let $m$ be an integer that is invertible in $K$.
Let $X$ be a proper snc scheme over $K$ and let  $g:C\to X$ be a non-constant morphism from a smooth proper curve $C$. 
Then for any  $\alpha\in H^i_{nr}(X/K,\mu_m^{\otimes j})$ and any non-zero rational function $\phi\in K(C)$, we have
$$
\langle g_\ast \operatorname{div}(\phi),\alpha \rangle =0 ,
$$
where $\operatorname{div}(\phi)\in \Div(C)$ denotes the divisor of zeros and poles of $\phi$.
\end{proposition}

\begin{proof}
This is an immediate consequence of Lemma \ref{lem:pairing-snc} and Proposition \ref{prop:pairing}.
\end{proof}

\begin{corollary} \label{cor:pairing-snc}
The pairing (\ref{def:pairing-snc}) descends to a bilinear pairing
$$
\CH_0(X)\times H^i_{nr}(X/K,\mu_m^{\otimes j})\longrightarrow H^i(K,\mu_m^{\otimes j}) 
$$
on the level of Chow groups.
\end{corollary}

\begin{proof}
This is a direct consequence of Proposition \ref{prop:pairing-snc}.
\end{proof}

\subsection{A pairing on the level of correspondences} \label{subsec:pairing-snc-Gamma}
Let $X$ and $Y$ be proper reduced algebraic schemes over a field $k$ and assume that $Y$ is an snc scheme.
Let further $m$ be an integer that is invertible in $k$.
We aim to define a bilinear pairing
\begin{align} \label{def:pairing-Gamma-snc}
Z_{\dim X}(X\times Y)\times H^i_{nr}(Y/k,\mu_m^{\otimes j})\longrightarrow \bigoplus_{l\in I} H^i (k(X_l) ,\mu_m^{\otimes j}),\ \ (\Gamma,\alpha)\mapsto ((\Gamma^\ast \alpha)_l)_{l\in I} ,
\end{align}
where $X_l$ with $l\in I$ denote the irreducible components of $X$.
It suffices to define $(\Gamma^\ast \alpha)_l$ for each $l\in I$, i.e.\ the composition of (\ref{def:pairing-Gamma-snc}) with the natural projection to  $H^i (k(X_l) ,\mu_m^{\otimes j}) $.
To this end, fix $l\in I$ and note that 
flat pullback induces a natural map 
$$
Z_{\dim X}(X\times Y)\longrightarrow Z_0(Y_{k(X_l)}) ,
$$ 
which descends to the level of Chow groups, i.e.\ sends cycles rationally equivalent to 0 on $X\times Y$ to cycles rationally equivalent to 0 on $Y_{k(X_l)}$.
In addition, there is a natural map $H^i_{nr}(Y/k,\mu_m^{\otimes j})\to H^i_{nr}(Y_{k(X_l)}/k(X_l),\mu_m^{\otimes j})$.
Using this, we define $$
(\Gamma,\alpha)\mapsto (\Gamma^\ast \alpha)_l ,
$$
by asking that the diagram
\begin{align} \label{diag:pairing-Gamma-snc}
\xymatrix{
Z_{\dim X}(X\times Y)\times H^i_{nr}(Y/k,\mu_m^{\otimes j}) \ar[rd]\ar[dd]& \\
 & \bigoplus_l  H^i (k(X_l) ,\mu_m^{\otimes j})\\
Z_0(Y_{k(X_l)})\times \bigoplus_l H^i_{nr} (Y_{k(X_l)}/k(X_l) ,\mu_m^{\otimes j}) \ar[ru]& 
,}
\end{align}
is commutative, where the lower horizontal arrow is induced by (\ref{def:pairing-snc}).

\begin{corollary} \label{cor:pairing-Gamma-snc}
Let $X$ and $Y$ be proper reduced algebraic schemes over a field $k$ and assume that $Y$ is an snc scheme.
Let further $m$ be an integer that is invertible in $k$.
Then the pairing (\ref{def:pairing-Gamma-snc}) descends to a well-defined pairing
\begin{align} \label{def:pairing-Chow-snc}
\CH_{\dim X}(X\times Y)\times H^i_{nr}(Y/k,\mu_m^{\otimes j})\longrightarrow \bigoplus_{l\in I} H^i (k(X_l) ,\mu_m^{\otimes j}).
\end{align}
\end{corollary}
\begin{proof}
Since (\ref{diag:pairing-Gamma-snc}) commutes, Corollary \ref{cor:pairing-snc} implies that  $\Gamma^\ast \alpha=0$ whenever $\Gamma\in Z_{\dim X}(X\times Y)$ is rationally equivalent to 0.
This proves the corollary.
\end{proof}
 
 \begin{remark}
The above corollary generalizes Corollary \ref{cor:pairing-Gamma} in two ways: $X$ and $Y$ may be reducible and $X$ may be arbitrarily singular.
\end{remark}

\section{Decompositions of the diagonal} \label{sec:diagonal}
The following notion goes back to Bloch \cite{bloch} (using an idea of Colliot-Th\'el\`ene) and Bloch--Srinivas \cite{BS}, and has for instance been studied in \cite{ACTP} and \cite{voisin}.

\begin{definition}
An algebraic scheme $X$ of pure dimension $n$ over a field $k$ admits a decomposition of the diagonal if 
\begin{align} \label{def:dec-of-diag}
[\Delta_X]=[X\times z]+[Z_X]\in \CH_n(X\times_k X) ,
\end{align}
where $\Delta_X\subset X\times_kX$ denotes the diagonal, $z\in  Z_0(X)$ is a 0-cycle on $X$ and $Z_X$ is a cycle on $X\times_k X$ which does not dominate any component of the first factor.
\end{definition}

For instance, $X=\CP^n_k$ admits a decomposition of the diagonal, because $\CH_n(\CP^n \times_k\CP^n)$ is generated by $[\CP^n\times \{x\}]$ for some $k$-rational point $x\in \CP^n$ together with cycles that do not dominate the first factor (namely $h^{n-i}\times h^i$ for $i= 1,\dots ,n$, where $h^i\subset \CP^n$ denotes a linear $i$-dimensional subspace).

To give a reducible example with a decomposition of the diagonal, we will now show that the union $X=\CP^n_k\cup_H \CP^n_k$ of two copies of $\CP^n_k$ glued along a hypersurface $H\subset \CP^n_k$ admits a decomposition of the diagonal as long as $H$ admits a $k$-rational point.
Indeed, if we write $X=X_1\cup X_2$, where $X_1,X_2$ are the irreducible components of $X$, then $X\times_kX$ has the 4 irreducible components $X_i\times_kX_j$ with $i,j\in \{1,2\}$. 
The class of the diagonal $\Delta_X$ may then be written as
$$
[\Delta_X]=[\Delta_{X_1}]+[\Delta_{X_2}]\in \CH_n(X\times X),
$$
where $\Delta_{X_i}\subset X_i\times X_i$ denotes the diagonal.
Because of the proper pushforward map 
$$
 \CH_n(X_i\times X_i)\longrightarrow  \CH_n(X\times X)
$$
for each $i=1,2$, we may use the decomposition of the diagonal of $X_i$ for each $i$ constructed above to get an identity
$$
\Delta_X=[X_1\times x_1]+[X_2\times x_2]+[Z_X] \in \CH_n(X\times X) ,
$$
where $x_i\in X_i$ is a $k$-rational point of $X_i$ and $Z_X$ is a cycle on $X\times_k X$ which does not dominate any component of the first factor.
This decomposition has not yet the form required in (\ref{def:dec-of-diag}).
However, since $H$ contains a $k$-rational point, any two $k$-rational points of $X$ are rationally equivalent (they may be joined by a chain of two lines).
This implies in particular $[X_2\times x_1 ]=[X_2\times x_2 ]$ and so
$$
\Delta_X=[X\times x_1]+[Z_X] \in \CH_n(X\times X).
$$
This gives a decomposition of the diagonal $\Delta_X$ as in (\ref{def:dec-of-diag}), as claimed.

Since in the above definition, $X$ is an arbitrary algebraic $k$-scheme, it may be illustrative to also consider a slightly more exotic example, such as $X=\Spec \C$ over the ground field $k=\R$.
In this case, the choice of a root of $-1$ in $\C$ identifies $X\times_{\R}X$ to the disjoint union of two copies of $\Spec \C$ (the Galois group of $X/\R$ permutes the two copies).
The diagonal $\Delta_X$ is isomorphic to $\Spec \C$ (it corresponds to one of the two components once we fixed a root of $-1$).
Hence, $\Delta_X$ is a 0-cycle of degree $2$ on the proper $\R$-scheme $X\times_\R X$.
On the other hand, any $0$-cycle $z$ on $X$ is just a multiple of $X$ and so the cycle $X\times z$ will have degree divisible by $4$.
This shows that no decomposition as in (\ref{def:dec-of-diag}) exists (where we note that $Z_X$ needs to be empty, as any nonzero cycle on $X\times_{\R} X$ dominates both factors).  

The following lemma generalizes the observation made in this last example, by showing that a proper algebraic scheme $X$ with a decomposition of the diagonal is geometrically connected and of index one,  
 i.e.\ $X$ contains a 0-cycle of degree one.

\begin{lemma}
Let $X$ be a proper algebraic scheme  of pure dimension $n$ over a field $k$.
If $X$ admits a decomposition of the diagonal as in  (\ref{def:dec-of-diag}), then the following holds:
\begin{enumerate}
\item the 0-cycle $z$ on $X$ in (\ref{def:dec-of-diag}) has degree one;\label{item:deg=1}
\item $X$ is geometrically connected. \label{item:geom-connected}
\end{enumerate}
\end{lemma}

\begin{proof}
Since $X$ is proper, the projection $\pr_1:X\times_kX\to X$ to the first factor gives rise to a proper pushforward map 
$$
(\pr_1)_\ast:\CH_n(X\times_kX)\longrightarrow \CH_n(X).
$$
If a decomposition as in  (\ref{def:dec-of-diag}) exists, then
$$
(\pr_1)_\ast[\Delta_X]=(\pr_1)_\ast [X\times z]=\deg(z)\cdot [X]\in \CH_n(X).
$$
Since $(\pr_1)_\ast[\Delta_X]=[X]\in \CH_n(X)$ and because this class is non-torsion (in fact $\CH_n(X)$ is a free abelian group on the irreducible components of the reduced scheme $X^{\red}$), the above identity implies $\deg(z)=1$, as we want. 
This proves (\ref{item:deg=1}).

If $X$ admits a decomposition of the diagonal, then for any  field extension $K/k$, the algebraic $K$-scheme $X_K=X\times_kK$ admits a decomposition of the diagonal as well, simply  by taking flat pullbacks of  (\ref{def:dec-of-diag}).
To prove (\ref{item:geom-connected}), it thus suffices to show that $X$ is connected if it admits a decomposition of the diagonal.
For a contradiction, assume that $X=\sqcup_{i=1}^m X_i$ is the disjoint union of finitely many algebraic $k$-schemes $X_1,\dots ,X_r$ with $r\geq 2$. 
We may then write the 0-cycle $z$ from (\ref{def:dec-of-diag}) in the form $z=\sum_{i=1}^rz_i$ such that $z_i$ is supported on $X_i$.
Pulling back the equality  (\ref{def:dec-of-diag}) to $X_i\times_k X_j$  for $i\neq j$ (this is possible because $X_i\times_k X_j$ is an open subscheme of $X\times_kX$ and so the inclusion is flat), we find that
$$
X_i\times z_j\in \CH_n(X_i\times_kX_j)
$$ 
is rationally equivalent to a cycle that does not dominate $X_i$ via the first projection.
Pushing this identity forward to the first factor, we find that $z_j$ has degree zero.
Since $r\geq 2$, this holds for all $j$ and so the 0-cycle $z$ has degree zero, contradicting item (\ref{item:deg=1}).  
This concludes the proof.
\end{proof}

In the case of varieties, we have the following important result.

\begin{lemma} \label{lem:dec-of-diag-CH_0}
A variety $X$ over a field $k$ admits a decomposition of the diagonal if and only if  there is a 0-cycle $z\in Z_0(X)$ on $X$ such that
\begin{align} \label{eq:dec-of-diag-CH_0}
[\delta_X]=[z_{K}]\in \CH_0(X_{K}), 
\end{align}
where $K=k(X)$ and where $\delta_X$ denotes the 0-cycle on $X_K$ that is induced by the diagonal $\Delta_X$.
\end{lemma}
\begin{proof}
Let $n:=\dim X$.
There is a natural isomorphism 
$$
 \lim_{\substack{\longrightarrow \\ \emptyset \neq U\subset X}} \CH_n(U\times_kX)\stackrel{\cong}\longrightarrow \CH_0(X_{K}).
$$
Using this, a decomposition of the diagonal (\ref{def:dec-of-diag}) implies directly an identity as in (\ref{eq:dec-of-diag-CH_0}) and the converse follows from the localization sequence \cite[Proposition 1.8]{fulton}.
\end{proof}

The following theorem which in the smooth proper case is due to Merkurjev, relates the notion of decompositions of the diagonal with unramified cohomology, cf.\ \cite[Theorem 2.11]{merkurjev}.

\begin{theorem}  \label{thm:dec->H_nr=0}
Let $k$ be a field and let $m$ be a positive integer that is invertible in $k$.
Let $X$ be a proper snc scheme (e.g.\ a smooth proper variety) over $k$.
If $X$ admits a decomposition of the diagonal, then the natural morphism
$$
\iota:H^i(k,\mu_m^{\otimes j})\longrightarrow H^i_{nr}(X/k,\mu_m^{\otimes j})
$$
is surjective for all $i$.
In particular, $ H^i_{nr}(X/k,\mu_m^{\otimes j})=0$ for all $i>0$ if $k$ is algebraically closed.
\end{theorem}
\begin{proof}
Assume that $X$ admits a decomposition of the diagonal as in (\ref{def:dec-of-diag}) and let $X_l$ with $l\in I$ denote the components of $X$.
By Corollary \ref{cor:pairing-Gamma-snc}, the pairing (\ref{def:pairing-Gamma-snc}) descends to a well-defined pairing
\begin{align} \label{def:pairing-Chow-snc-2}
\CH_{\dim X}(X\times X)\times H^i_{nr}(X/k,\mu_m^{\otimes j})\longrightarrow \bigoplus_{l\in I} H^i (k(X_l) ,\mu_m^{\otimes j}).
\end{align}
It follows form the definition of this pairing  in (\ref{def:pairing-Gamma-snc}) that
$$
[\Delta_X]^\ast \alpha=\alpha
$$
for all 
$$
\alpha\in H^i_{nr}(X/k,\mu_m^{\otimes j})\subset \bigoplus_{l\in I} H^i (k(X_l) ,\mu_m^{\otimes j}) .
$$
On the other hand
$$
[Z_X]^\ast \alpha=0 ,
$$
whenever $Z_X\in \CH_{\dim X}(X\times X)$ does not dominate any component of the first factor.
Hence, (\ref{def:dec-of-diag})  implies
$$
\alpha=[\Delta_X]^\ast \alpha=[X\times z]^\ast \alpha+[Z_X]^\ast \alpha=[X\times z]^\ast \alpha .
$$
If $z=\sum_s a_s[x_s]$ for some integers $a_s$ and closed points $x_s\in X$ with structure morphisms $f_{x_s}:\Spec \kappa(x_s)\to \Spec k$, then
$$
[X\times z]^\ast \alpha=\iota(\sum_s a_s(f_{x_s})_\ast \alpha|_{x_s}) ,
$$
where 
$$
\iota:H^i(k,\mu_m^{\otimes j})\longrightarrow H^i_{nr}(X/k,\mu_m^{\otimes j})
$$
is the natural morphism.
This proves the theorem.
\end{proof}

\subsection{Connection to rationality and stable birational types}

For the following result, see \cite[Lemme 1.5]{CTP} in the smooth case and \cite[Lemma 2.4]{Sch-JAMS} in general.

\begin{lemma} \label{lem:ratl->decofdiag}
A variety $X$ over a field $k$ that is stably rational  (or more generally retract rational) admits a decomposition of the diagonal.
\end{lemma}

\begin{proof} 
Recall that $X$ is retract rational if there are rational maps $f:X\dashrightarrow \CP^N_k$ and $g:\CP^N_k\dashrightarrow X$ with $g\circ f=\id$ (here we ask implicitly that the composition $g\circ f$ is defined).
This condition is for instance satisfied if $X$ is stably rational. 
We denote by $\Gamma_f\subset X\times \CP^N_k$ and $\Gamma_g\subset \CP^N_k\times X$ the closure of the graph of $f$ and $g$, respectively.
Replacing $X$ by a projective model, we may assume that $X$ is proper over $k$.
Replacing $X$ by $\Gamma_f$, we may also assume that $f$ is a morphism, which is automatically proper since $X$ is proper.
For any field extension $K$ of $k$, we then get morphisms
$$
f_\ast:\CH_0(X_K)\longrightarrow \CH_0(\CP^N_K)\ \ \text{and}\ \ g^\ast:\CH_0(\CP^N_K)\longrightarrow \CH_0(X_K),
$$
where $g^\ast$ is defined by pulling back cycles from $\CP^N_K$ to $(\Gamma_f)_K$ (using \cite[Definition 8.1.2]{fulton}, which works because $\CP^N_K$ is smooth) and pushing these cycles forward to $X_K$ via the natural proper morphism $(\Gamma_f)_K\to X_K$. 

Let now $K=k(X)$ be the function field of $X$.
Then 
$$
[\delta_X]=g^\ast \circ f_\ast [\delta_X] ,
$$
because $g\circ f=\id$ implies $g^\ast\circ f_\ast=\id$.
On the other hand, $\CH_0(\CP^N_K)\cong \Z$ is generated by $[z_K]$ for any $k$-point $z$ of $\CP^N_k$.
Since $f_\ast [\delta_X]$ has degree one, it follows that $f_\ast[ \delta_X]=[z_K]$ and so
$$
[\delta_X]=g^\ast \circ f_\ast [ \delta_X]=g_\ast [z_K]=[(g_\ast z)_K] ,
$$ 
which shows that $X$ admits a decomposition of the diagonal by Lemma \ref{lem:dec-of-diag-CH_0}. 
\end{proof}

The above lemma implies that a variety that does not admit a decomposition of the diagonal cannot be stably rational.
Motivated by work of Shinder \cite{Shinder}, the following generalization is proven in \cite[Theorem 1.1]{Sch-EPIGA}.

\begin{theorem} Let $k$ be an uncountable algebraically closed field and let $X$ be a smooth projective $k$-variety that does not admit a decomposition of the diagonal.
Then for any positive integers $n$ and $d$, $X$ is not stably birational to a very general hypersurface of degree $d$ and dimension $n$ over $k$.
\end{theorem}

\begin{remark}
Let  $B:=\CP H^0(\CP^{n+1}_k,\mathcal O(d))$ and let $\mathcal Y\to B$ be the universal degree $d$ hypersurface in $\CP^{n+1}_k$.
The conclusion of the above theorem means that there is a countable union (that depends on $X$)  of proper closed subsets of  $B $ such that any hypersurface $Y$, that corresponds to a point outside this countable union, is not stably birational to $X$.
\end{remark}

The above result is a consequence of Theorem \ref{thm:degeneration:mixed:char:1} below (see \cite[Theorem 4.1]{Sch-EPIGA}).
To state the result,  we need to recall the following terminology.

Firstly, for a discrete valuation ring $R$, a proper flat $R$-scheme $\mathcal X$ is called strictly semi-stable, if the special fibre $X_0$ is a geometrically reduced simple normal crossing divisor on $\mathcal X$.
That is, the components of $X_0$ are smooth Cartier divisors on $\mathcal X$ and the scheme-theoretic intersection of $r$ different components of $X_0$ is either empty or smooth and equi-dimensional of codimension $r$ in $\mathcal X$.

Secondly, a proper variety $X$ over a field $k$ has universally trivial Chow group of 0-cycles if for any field extension $K/k$, the degree map $\deg:\CH_0(X_K)\to \Z$ is an isomorphism.
If $X$ is smooth and proper over $k$, then this is equivalent to (\ref{eq:dec-of-diag-CH_0}) and hence equivalent to the fact that $X$ admits a decomposition of the diagonal.

\begin{theorem} \label{thm:degeneration:mixed:char:1}
Let $R$ be a discrete valuation ring with algebraically closed residue field $k$.
Let $\pi:\mathcal X\to \Spec R$ and $\pi':\mathcal X'\to \Spec R$ be flat projective morphisms with geometrically connected fibres such that $\pi$ is strictly semi-stable and $\pi'$ is smooth.
Assume 
\begin{enumerate}
\item the special fibre $X_0$ of $\pi$ has universally trivial Chow group of 0-cycles;
\item the special fibre $X'_0$ of $\pi'$ does not admit a decomposition of the diagonal.
\end{enumerate}
Then the geometric generic fibres of $\pi$ and $\pi'$ are not stably birational to each other.
\end{theorem}

\subsection{Torsion orders} \label{subsec:Tor(X)}

Let $X$ be a rationally chain connected proper variety over a field $k$ (e.g.\ a smooth Fano variety, see e.g.\ \cite[Theorem V.2.13]{kollar2}).
This means that any two points of the base change $X_{\overline L}$ of $X$ to any algebraically closed field $\overline L\supset k$ can be joined by a chain of rational curves (if $k$ is uncountable then it suffices to ask this for $\overline L=\overline k$).  
In particular, the degree  map 
$$
\deg:\CH_0(X_{\overline k})\longrightarrow \Z
$$
is an isomorphism and this property remains true if we replace $\overline k$ by a larger algebraically closed field.
We thus find that for any 0-cycle $z\in \CH_0(X)$ of degree one, the 0-cycle
$$
\delta_X-z_{k(X)}\in \CH_0(X_{k(X)})
$$
lies in the kernel of the natural map
$$
\CH_0(X_{k(X)})\longrightarrow \CH_0(X_{\overline {k(X)}}).
$$
But then there must be a finite extension $L$ of $k(X)$, so that the above 0-cycle vanishes already in $ \CH_0(X_{L})$.
Since the natural composition
$$
\CH_0(X_{k(X)})\longrightarrow  \CH_0(X_{L})\longrightarrow \CH_0(X_{k(X)})
$$
is given by multiplication by $\deg (L/k)$, we find that 
$$
\delta_X-z_{k(X)}\in \CH_0(X_{k(X)})
$$
is a torsion class.
This motivates the following definition, that has for instance been studied in  \cite{CL}, \cite{kahn}, and \cite{Sch-torsion}.

\begin{definition} \label{def:Tor(X)}
Let $X$ be a proper variety over a field $k$ and let $\delta_X$ be the 0-cycle on $X_{k(X)}$ that is induced by the diagonal $\Delta_X$.
Then the torsion order $\Tor(X)\in \N\cup\{\infty\}$ of $X$ is  the smallest positive integer such that the 0-cycle $\Tor(X)\cdot \delta_X$ may be written as
$$
\Tor(X)\cdot \delta_X=z_{k(X)}\in \CH_0(X_{k(X)}) ,
$$
for some zero-cycle $z$ on $X$, and it is $\infty$ if no such integer exists.
\end{definition}

\begin{remark}
By Lemma \ref{lem:dec-of-diag-CH_0},  $\Tor(X)=1$ if and only if $X$ admits a decomposition of the diagonal and the same argument shows that for any positive integer $e$, $e\cdot \Delta_X$ admits a decomposition as in (\ref{def:dec-of-diag}) if and only if $\Tor(X)<\infty$ divides $e$. 
\end{remark}

If $k$  is algebraically closed and $X$ is smooth and proper, then the torsion order of $X$ bounds the order of any unramified cohomology class $\alpha$ on $X$ of cohomological degree at least $1$, as for $\deg \alpha\geq 1$ and $\Tor(X)<\infty$, 
$$
0=\langle \Tor(X)\cdot \delta_X-z_{k(X)},\alpha \rangle=\langle \Tor(X)\cdot \delta_X ,\alpha \rangle=\Tor(X)\cdot \alpha ,
$$
because $\langle z_{k(X)},\alpha \rangle=0$ as $\alpha$ restricts to zero on any closed point of $X$ (because $k$ is algebraically closed).

\begin{lemma} \label{lem:Tor(X)}
Let $f:X\to Y$ be a proper dominant morphism between proper $k$-varieties.
If $\Tor(X)<\infty$, then
$$
\Tor(Y)\mid \deg(f)\cdot \Tor(X) .
$$
\end{lemma}
\begin{proof} 
The morphism $f$ induces a proper morphism $f\times f:X\times X\to Y\times Y$ and hence a proper morphism $f':X_{k(X)}\to Y_{k(Y)}$ with
$$
0=f'_\ast(\Tor(X)\cdot \delta_X-z_{k(X)})=\deg(f)\cdot \Tor(X)\cdot \delta_{Y}-f_\ast z_{k(X)}\in \CH_0(Y_{k(Y)}).
$$
This implies
$
\Tor(Y)\mid \deg(f)\cdot \Tor(X) 
$, as we want.
\end{proof}

\begin{corollary}\label{cor:Tor(X)}
Let $Y$ be a proper $k$-variety of dimension $n$ and let $f:\CP^n_k\dashrightarrow Y$ be a dominant rational map.
Then $\Tor(Y)\mid \deg(f)$.
\end{corollary}
\begin{proof}
Let $ X\subset \CP^n_k\times Y$ be the closure of the graph of $f$.
Then $X$ is a proper variety over $k$ and the second projection yields a morphism $f':X\to Y$ of degree $\deg(f')=\deg(f)$.
Since $X$ is rational, it has torsion order 1 by Lemma \ref{lem:ratl->decofdiag} and so the corollary follows from Lemma \ref{lem:Tor(X)}.
\end{proof}

\section{Specialization method} \label{sec:specialization}

The specialization method that we explain in this section was initiated by Voisin \cite{voisin}, with important improvements by Colliot-Th\'el\`ene--Pirutka \cite{CTP} and later some further improvements by the author \cite{Sch-duke,Sch-JAMS}.

\begin{notation} \label{not:specialization}
Let $R$ be a discrete valuation ring with fraction field $K$ and algebraically closed residue field $k$.
Let $\pi:\mathcal X\to \Spec R$ be a proper flat $R$-scheme with connected fibres and denote by $X=\mathcal X\times \overline K$ and $Y=\mathcal X\times k$ the geometric generic fibre and special fibre of $\pi$.
\end{notation}

Fulton showed the following specialization result for Chow groups, see \cite[\S 4.4]{fulton2} and \cite[\S 20.3]{fulton}.

\begin{theorem} \label{thm:fulton}
In the notation \ref{not:specialization}, let $\mathcal X_\eta=\mathcal X\times_RK$.
Then we have for any integer $i$ a specialization map
$$
sp:\CH_i(\mathcal X_\eta)\longrightarrow \CH_i(Y) ,
$$
which is defined by 
$
sp(\gamma)=\overline \gamma|_{Y} ,
$ 
where $\overline \gamma$ denotes the closure of $\gamma$ in the total space $\mathcal X$ and where $\overline \gamma|_{Y} $ denotes its restriction to $Y$, i.e.\ the pullback of $\overline \gamma$ to the Cartier divisor $Y\subset \mathcal X$.  
\end{theorem} 
\begin{proof}
Since $R$ is a discrete valuation ring, the closure $\overline \gamma$ in $\mathcal X$ of a cycle $\gamma$ on the generic fibre $\mathcal X_\eta$ will automatically be  flat over $R$.
Hence the map in the theorem exists on the level of cycles and we only need to see that it descends to the Chow group modulo rational equivalence.
By \cite[p.\ 154, \S 1.9]{fulton}, the following sequence is exact
$$
\CH_{i+1}(Y)\stackrel{\iota_\ast} \longrightarrow\CH_{i+1}(\mathcal X)\stackrel{j^\ast}\longrightarrow \CH_i(\mathcal X_\eta)\longrightarrow 0 ,
$$
where $\iota:Y\to\mathcal X$  and $j:\mathcal X_\eta\to \mathcal X$ denote the natural morphisms.
Assume that $\gamma$ and $\gamma'$ are two rationally equivalent $i$-dimensional cycles on $\mathcal X_\eta$ with closures $\overline \gamma$ and $\overline \gamma'$ in $\mathcal X$.
Then $j^\ast[\overline \gamma]=j^\ast[\overline \gamma']$ and so
$$
[\overline \gamma]-[\overline \gamma']=\iota_\ast \xi
$$ 
for some $\xi\in \CH_{i+1}(Y) $.
On the other hand, $Y$ is a principal Cartier divisor on $\mathcal X$ that contains the support of $\iota_\ast \xi$ and so $
\iota^\ast( \iota_\ast \xi) =0
$ by the definition of $\iota^\ast$, see \cite[\S 1.7]{fulton2}.
Hence,
$
\iota^\ast[\overline \gamma]=\iota^\ast [\overline \gamma'] \in \CH_i(Y),
$
as we want. This concludes the proof.
\end{proof}

Fulton's theorem has the following consequence.

\begin{corollary}[\cite{voisin,CTP}] \label{cor:specialization-of-diagonal}
In the notation \ref{not:specialization}, if $X$ admits a decomposition of the diagonal and the special fibre $Y$ is pure-dimensional (e.g.\ this is automatic if $X$ is integral by Krull's Hauptidealsatz), then $Y$ admits a decomposition of the diagonal as well.
\end{corollary}
\begin{proof} 
Replacing $R$ by its completion we may assume that in the notation \ref{not:specialization} $R$ is complete.
Let $n:=\dim Y$.
Since $\pi$ is flat and $Y$ is pure-dimensional, $X$ is also pure-dimensional of dimension $n$.
Assume that there is a decomposition of the diagonal
$$
[\Delta_X]=[X\times z]+[Z_X]\in \CH_n(X\times_{\overline K} X) 
$$
of $X$.
This implies that there is a finite field extension $L/K$, so that the above relation holds over $L$.
Since $R$ is complete, the integral closure of $R$ in $L$ is a discrete valuation ring $R'$, see e.g.\ \cite[Th\'eor\`eme 23.1.5 and Corollaire 23.1.6]{EGAIV}.
Hence, replacing $\mathcal X\to \Spec R$ by the finite base change $\Spec R'\to \Spec R$, we may assume that the above relation is defined and holds already over the field $K$.
Applying the specialization map from Theorem \ref{thm:fulton}, we then find
$$
[\Delta_Y]=sp([\Delta_X])=[Y\times sp(z)]+sp[Z_X]\in \CH_n(Y\times_{k} Y),
$$
where $sp(z)\in \CH_0(Y)$ is a 0-cycle on $Y$.
Projecting the closure of $Z_X$ to the first factor yields a cycle on $\mathcal X$ that is automatically flat over $R$ and which has generically dimension at most $n-1$.
It follows that also the special fibre of this cycle has dimension at most $n-1$.
We deduce that $sp[Z_X]$ can be represented by a cycle on $Y\times_k Y$ whose image via the first projection has dimension at most $n-1$ and so it does not dominate any component of $Y$, as $Y$ is of pure dimension $n$.
Hence, $Y$ admits a decomposition of the diagonal, as we want.
This concludes the corollary.
\end{proof}

Together with Theorem \ref{thm:dec->H_nr=0}, we deduce the following criterion, which has in this form been proven by Colliot-Th\'el\`ene--Pirutka, thereby generalizing an earlier version of Voisin \cite{voisin}, who initiated the method. 

\begin{theorem}[{\cite[Th\'eor\`eme 1.12]{CTP}}] \label{thm:specialization-CTP} 
In the notation \ref{not:specialization}, assume that $Y$ is integral and that the following holds:
\begin{itemize}
\item $ H^i_{nr}(k(Y)/k,\mu_m^{\otimes j})\neq 0$ for some $i>0$ and some $j$;
\item there is a resolution $\tau:Y'\to Y$ such that $\tau_\ast: \CH_0(Y'_L)\to \CH_0(Y_L)$ is an isomorphism for all field extensions $L/k$.
\end{itemize} 
Then $X$ does not admit a decomposition of the diagonal.
\end{theorem}
\begin{proof}
For a contradiction, assume that $X$ admits a decomposition of the diagonal.
Then so does $Y$ by Corollary \ref{cor:specialization-of-diagonal}.
By Lemma \ref{lem:dec-of-diag-CH_0}, $\delta_Y=z_{k(Y)}$ holds in $\CH_0(Y_{k(Y)})$.
By assumption, $\tau_\ast: \CH_0(Y'_L)\to \CH_0(Y_L)$ is an isomorphism for $L=k(Y)$ and so we find $\delta_{Y'}=z'_{k(Y)}$ in $\CH_0(Y'_{k(Y)})$ and for some 0-cycle $z'$ on $Y'$.
Since $k(Y')=k(Y)$, this shows by Lemma \ref{lem:dec-of-diag-CH_0} that $Y'$ admits a decomposition of the diagonal, and so
 Theorem \ref{thm:dec->H_nr=0} yields
$$
H^i_{nr}(k(Y)/  k,\mu_m^{\otimes j})= 0,
$$
which contradicts our assumptions. 
This concludes the theorem.
\end{proof}

Thanks to our generalization of Merkurjev's pairing to the case of snc schemes in Section \ref{sec:snc}, we obtain immediately the following (new) variant of the above theorem. 

\begin{theorem} \label{thm:specialization} 
In the notation \ref{not:specialization}, assume that $Y$ is a snc scheme over $k$ with
$$
H^i_{nr}(Y /k,\mu_m^{\otimes j})\neq 0
$$
for some $i>0$ and some integer $m$ that is invertible in $k$.
Then $X$ does not admit a decomposition of the diagonal.
\end{theorem}
\begin{proof}
For a contradiction, assume that $X$ admits a decomposition of the diagonal.
Then so does $Y$ by Corollary \ref{cor:specialization-of-diagonal}.
Hence,  Theorem \ref{thm:dec->H_nr=0} yields
$$
H^i_{nr}(Y/  k,\mu_m^{\otimes j})= 0,
$$
which contradicts our assumptions. 
This concludes the theorem.
\end{proof}
  
To apply the above theorems to a given projective variety $X$, one has to construct a resolution of an integral degeneration $Y$ of $X$, or an $R$-model $\mathcal X$ of $X$ whose special fibre is an snc scheme.
On the other hand, the special fibre needs to have non-trivial unramified cohomology and practice shows that this usually forces the model $\mathcal X$ to be highly non-smooth.

The presence of the singularities in the special fibre makes it often very hard   (and sometimes practically impossible) to compute a model $\mathcal X$ or a resolution $\tau:Y'\to Y$ as in the above theorems.
This problem was solved in \cite{Sch-duke}, where it was shown that much more singular models can be used.
This method was generalized further in \cite[Proposition 3.1]{Sch-JAMS} (allowing alterations instead of resolutions), as follows.

\begin{theorem}[\cite{Sch-duke,Sch-JAMS}] \label{thm:specialization-schreieder}
In the notation \ref{not:specialization}, assume that $X$ and $Y$ are integral.
Let $m\geq 2$ be an integer that is invertible in $k$.  
Suppose that for some integers $i\geq 1$ and $j$ there is a class $\alpha \in H^i_{nr}(k(Y)/ k,\mu_m^{\otimes j})$ of order $m$  
 such that for any alteration $\tau:Y'\to Y$ and any (scheme) point $x\in Y'$ with $\tau(x)\in  Y^{\sing}$ we have
\begin{align} \label{cond:vanishing}
(\tau^\ast\alpha)|_{x}=0\in H^i(\kappa(x),\mu_m^{\otimes j}) .
\end{align}
Then $X$ does not admit a decomposition of the diagonal. 
\end{theorem}
 
\begin{remark}
The main point of the above result lies in the fact that in many situations of interest, condition (\ref{cond:vanishing}) is automatically satisfied (cf.\ Theorem  \ref{thm:vanishing} below); that is, one often
 gets the condition (\ref{cond:vanishing}) on the singularities for free, without even computing an alteration (or resolution).  
\end{remark} 
 
 \begin{remark}
 It is not hard to see that the proof that we give below still works in the case where the special fibre might be reducible and $Y$ in the above theorem is replaced by a reduced component of the special fibre, see  \cite[Proposition 6.1]{Sch-torsion} for more details.  
 \end{remark}

\begin{remark}
The proof will show more generally that  
 the torsion order of $X$ is either infinite or divisible by $m$, see Definition \ref{def:Tor(X)}.
\end{remark}

\begin{proof}[Proof of Theorem \ref{thm:specialization-schreieder}]
Let $A=\mathcal O_{\mathcal X,Y}$ be the local ring of $\mathcal X$ at the generic point of $Y$.
Since $Y$ is reduced, it follows that $A$ is a discrete valuation ring with residue field $k(Y)$.
Since $X$ admits a decomposition of the diagonal, so does $Y$ by Corollary \ref{cor:specialization-of-diagonal}.
We can restrict this identity in the Chow group of $Y\times Y$ to $Y_{k(Y)}$ and get an equality
$$
\delta_{Y}=z_{k(Y)}\in \CH_0(Y_{k(Y)})
$$
where $z$ is a 0-cycle on $Y$ and $\delta_Y$ denotes the 0-cycle that is induced by the diagonal.

By Gabber's improvement  of de Jong's alteration theorem \cite{deJong}, there is an alteration $\tau:Y'\to Y$ whose degree is coprime to $m$, see \cite[Theorem 2.1]{IT}.
We would like to pull back the above relation to the Chow group of 0-cycles of $Y'_{k(Y)}$.
In general this is impossible if $Y_{k(Y)}$ is not smooth.
Instead we can restrict the above equality to the smooth locus of $Y_{k(Y)}$ (using flat pullbacks) and pulling back this identity to the Chow group of 0-cycles on $\tau^{-1}(Y^{\sm}_{k(Y)})$.
By the localization exact sequence (see \cite[Proposition 1.8]{fulton}), we then get an identity
\begin{align} \label{eq:delta_tau}
\delta_{\tau}=z_{k(Y')}+z'\in \CH_0(Y'_{k(Y)}) ,
\end{align}
where $\delta_\tau$ is the 0-cycle on $Y'_{k(Y)}$ that is induced by the graph of $\tau$ inside $Y'\times Y$, $z\in \CH_0(Y')$ is a 0-cycle on $Y'$ and $z'\in \CH_0(Y'_{k(Y)})$ is a 0-cycle whose support satisfies
$$
\supp(z')\subset \tau^{-1}(Y^{\sing})_{k(Y)}.
$$ 
The definition of the pairing in (\ref{def:pairing}) thus shows by (\ref{cond:vanishing}) that
$
\langle  z',\tau^\ast \alpha\rangle=0.
$
Since  $k$ is algebraically closed and so $H^i(k,\mu_m^{\otimes j})=0$, we also have
$
\langle z_{k(Y')},\tau^\ast \alpha\rangle=0.
$
By linearity, we then get
$$
\langle z_{k(Y')}+z',\tau^\ast \alpha\rangle=0.
$$
Using (\ref{eq:delta_tau}), the above relation in $ \CH_0(Y_{k(Y)}) $ thus shows by Corollary \ref{cor:pairing} that
$$
\langle \delta_{\tau},\tau^\ast \alpha\rangle=0.
$$
On the other hand, the definition of the pairing in (\ref{def:pairing}) yields
$$
\langle \delta_{\tau},\tau^\ast \alpha\rangle=\langle \tau_\ast \delta_{\tau}, \alpha\rangle=\langle \deg \tau \cdot \delta_Y, \alpha\rangle=\deg \tau\cdot \alpha.
$$
Hence,
$$
\deg \tau\cdot \alpha=0\in H^i(k(Y),\mu_m^{\otimes j}) ,
$$
which contradicts the fact that $\alpha$ is nonzero and $\deg \tau$ is coprime to $m$.
This concludes the proof of the theorem.
\end{proof} 

Theorems \ref{thm:specialization-CTP}, \ref{thm:specialization} and \ref{thm:specialization-schreieder} yield together with Lemma \ref{lem:ratl->decofdiag} a powerful method to obstruct rationality of varieties that have interesting specializations.
Using entirely different methods, Nicaise--Shinder \cite{NS} and Kontsevich--Tschinkel \cite{KT} recently showed the following related result. 
 
 \begin{theorem}[\cite{NS,KT}] \label{thm:NS-KT}
 In the notation \ref{not:specialization}, assume that $k$ has characteristic 0, that $\mathcal X$ is regular and that $Y$ is an snc scheme.
Let $Y_l$ with $l\in I$ denote the irreducible components of $Y$ and write $Y_J:=\bigcap_{l\in J} Y_l$ for any $J\subset I$.
If
\begin{align} \label{eq:thm:NS-KT}
\sum_{\emptyset\neq J\subset I} (-1)^{|J|-1} [Y_J\times \CP_k^{|J|-1}] \neq [\CP_k^{\dim X}]
\end{align}
holds true in the free abelian group on (stable) birational equivalence classes of smooth $k$-varieties, 
then $X$ is not (stably) rational. 
 \end{theorem}

\begin{remark} 
In order to show that (\ref{eq:thm:NS-KT}) holds, one has to show in practice that at least one of the terms $Y_J\times \CP_k^{|J|-1}$ is not (stably) rational and that it is not cancelled out in the alternating sum by the remaining terms.
This strategy has recently been implemented successfully by Nicaise--Ottem \cite{NO}, who use as an input known stable irrationality results, which in turn are proven by applications of Theorems \ref{thm:specialization-CTP} or \ref{thm:specialization-schreieder}.
\end{remark}

\section{Examples with nontrivial unramified cohomology} \label{sec:examples}

In this section we collect some of the most important known constructions of rationally connected varieties $Y$ over algebraically closed fields with nontrivial unramified cohomology.
All examples have the following approach in common.   

\begin{enumerate}
\item Start with a smooth projective rational variety $S$ and a nontrivial class $\alpha\in H^i(k(S),\mu_m^{\otimes j})$.
\item Construct another variety $Y_{\alpha}$ (usually of larger dimension) together with a dominant morphism $f_\alpha:Y_\alpha \to S$ such that $f_\alpha^\ast \alpha =0$.
This usually requires that the ramification locus of $Y_\alpha$ is contained in the ramification locus of $\alpha$ and often both loci will coincide.
\item Construct a rationally connected variety $Y$ with a dominant morphism $f:Y\to S$ such that:
\begin{enumerate}[(i)]
\item \'etale locally at any (codimension one) point of $S$ where $\alpha$ has nontrivial residue, the fibration $f$ coincides with the fibration $f_\alpha$ up to birational equivalence; \label{item:construction:3i}
\item Zariski locally over $S$, the fibrations $f$ and $f_\alpha$ are not birationally equivalent.\label{item:construction:3ii}
\end{enumerate} 
\item Show that $f^\ast \alpha \in H^i(k(Y),\mu_m^{\otimes j})$ is unramified over $k$ by exploiting (\ref{item:construction:3i}). 
The idea is that $f^\ast \alpha$ can possibly only have residues above points on $S$ where $\alpha$ ramifies, but condition (\ref{item:construction:3i}) ensures that \'etale locally at such points,  $f^\ast \alpha$ is 0 and so it must have trivial residue. 
\label{item:construction:4}
\item Find a reason why $f^\ast \alpha$ is nonzero; this is a tricky point, because $f^\ast \alpha$ will be unramified by item (\ref{item:construction:4}) and so it is a priori impossible to check nontriviality by a residue computation.
An obviously necessary condition here is condition (\ref{item:construction:3ii}). 
\label{item:construction:5}
\end{enumerate}

\begin{remark}
In \cite{Sch-JAMS} and \cite{Sch-torsion}, it was shown that the following flexible condition ensures the nonvanishing of $f^\ast \alpha$: $f^\ast \alpha$ is nonzero if there is a degeneration $Y_0\to S$ of $Y\to S$, so that there is a $k(S)$-rational point in the smooth locus of the generic fibre of $Y_0\to S$.
\end{remark}

\begin{remark}
We emphasize that conditions (\ref{item:construction:3i}) and (\ref{item:construction:3ii}) do not automatically imply that the aim in items (\ref{item:construction:4}) and (\ref{item:construction:5}) can be achieved, it should rather be seen as a guideline how to find potential candidates for which one might hope to be able to prove (\ref{item:construction:4}) and (\ref{item:construction:5}). 
\end{remark}

\subsection{Quadric bundles à la  Artin--Mumford  and  Colliot-Thélène--Ojanguren} \label{subsec:CTO}
The starting point here is the following result of Arason \cite{arason} and Orlov--Vishik--Voevodsky, see \cite[Theorem 2.1]{OVV}.

\begin{theorem} \label{thm:OVV}
Let $K$ be a field of characteristic 0, let $a_1,\dots ,a_n\in K^\ast$ and consider the associated symbol $\alpha=(a_1,\dots ,a_n)\in H^n(K,\mu_m^{\otimes n})$.
Consider the associated Pfister quadric $f:Q_{\alpha}\longrightarrow \Spec K$ over $K$, given by  
$$
Q_\alpha := \left\lbrace  \sum_{\epsilon\in \{0,1\}^n} (-a_1)^{\epsilon_1}(-a_2)^{\epsilon_2}\dots (-a_n)^{\epsilon_n}\cdot z_{\rho(\epsilon)}^2=0  \right\rbrace \subset  \CP^{2^n-1}_K
$$
where $\epsilon=(\epsilon_1,\dots ,\epsilon_n)\in \{0,1\}^n$, and where  $\rho:\{0,1\}^n\to \{0,1,2,\dots ,2^n-1\}$ denotes the bijection $\rho(\epsilon)=\sum_{i=0}^{n-1}\epsilon_i\cdot 2 ^i $.
Then,
$$
\ker \left(f^\ast: H^n(K,\mu_2^{\otimes n})\longrightarrow H^n(K(Q),\mu_2^{\otimes n})\right) =\left\lbrace 0, \alpha \right\rbrace . 
$$
\end{theorem}

In the construction of Artin--Mumford \cite{artin-mumford} and Colliot-Thélène--Ojanguren \cite{CTO}, one considers fibrations $f:Y\to \CP^n_k$ whose generic fibre is (stably) birational to a  Pfister quadric $Q_\alpha$ over $K=k(\CP^n)$ with the following special property.

\begin{definition}
Let $k$ be a field of characteristic 0 and let $K=k(\CP^n)$ for some integer $n\geq 2$.
Let further $a_1,\dots ,a_{n-1},b_1,b_2\in K^\ast$ be nonzero rational functions on $\CP^n_k$ and consider the symbols
$$
\alpha_j:=(a_1,a_2,\dots ,a_{n-1},b_j)\in H^n(K,\mu_2^{\otimes n})
$$
for $j=1,2$.
Then the Pfister quadric $Q_{\alpha_1+\alpha_2}$ associated to $\alpha=\alpha_1+\alpha_2$ is called CTO- type quadric, if the following holds:
\begin{enumerate}[(1)]
\item For any geometric valuation $\nu$ on $K$ over $k$, we have $\del_\nu \alpha_j=0$ for at least one $j\in \{1,2\}$.
In other words, there is no such valuation $\nu$ such that $\alpha_1$ and $\alpha_2$ have both nontrivial residue with respect to $\nu$.
\item $\alpha_j\neq 0$ for $j=1,2$ (if $k=\overline k$, this is equivalent to asking that for each $j=1,2$, $\alpha_j$ has at least one nontrivial residue at some geometric valuation $\nu$). 
\end{enumerate}
\end{definition}

With this definition and the above theorem, it is easy to prove the following result that goes back to Colliot-Thélène--Ojanguren, see \cite[Proposition 17]{Sch-duke}.

\begin{proposition}[\cite{CTO}]
Let $k$ be a field of characteristic 0 and let $Y$ be a $k$-variety with a dominant morphism $f:Y\to \CP^n_k$, such that the generic fibre of $f$ is stably birational to a CTO-type quadric $Q_{\alpha_1+\alpha_2}$ over the function field $k(\CP^n)$.
Then 
$$
0\neq f^\ast \alpha_1=f^\ast \alpha_2\in H^n_{nr}(k(Y)/k,\mu_2^{\otimes n}) .
$$
\end{proposition}
\begin{proof}
By Theorem \ref{thm:OVV}, $f^\ast (\alpha_1+\alpha_2)=0$ and so $f^\ast \alpha_1=f^\ast \alpha_2$ (as we work with mod $2$ coefficients).
Moreover, $f^\ast \alpha_1=0$ would by Theorem \ref{thm:OVV} imply $\alpha_1=0$ or $\alpha_1=\alpha_1+\alpha_2$ (hence $\alpha_2=0$), which contradicts the assumption $\alpha_j\neq 0$ for $j=1,2$.
Hence,
$$
0\neq f^\ast \alpha_1=f^\ast \alpha_2\in H^n (k(Y) ,\mu_2^{\otimes n})
$$
and it suffices to show that this class is unramified over $k$.
For this, let $\nu$ be a geometric  valuation of $k(Y)$ over $k$ and consider the restriction $\mu:=\nu|_K$ of $\nu$ to $ k(\CP^n)$.
If $\mu$ is trivial, then $\del_\nu f^\ast \alpha_j=0$ is clear.
Otherwise, $\mu$ is a geometric valuation on $k(\CP^n)$ over $k$.
By the definition of CTO-type quadrics, there is some $j\in \{1,2\}$ with
$
\del_{\mu}\alpha_j=0
$
and so $\del_\nu f^\ast \alpha_j=0$ follows from the diagram (\ref{diag:del-pullbacks}).
This concludes the proof of the proposition.
\end{proof}

The main difficulty in this approach is the construction of CTO-type quadrics.
The example of Artin--Mumford \cite{artin-mumford} is a conic that is stably birational to a CTO-type quadric surface over $k(\CP^2)$.
CTO-type quadrics over $k(\CP^3)$ have been constructed by Colliot-Thélène--Ojanguren in \cite{CTO} and the general case of CTO-type quadrics over $k(\CP^n)$ for arbitrary $n\geq 2$ was established in \cite[Section 6]{Sch-duke}.

An algebraic variant of this construction which leads to fibrations over rational bases whose generic fibres are products of certain Pfister quadrics was established by Peyre \cite{peyre} and Asok \cite{asok}.
This approach allows generalizations to $\mu_\ell$-coefficients for any prime $\ell$, but it still relies heavily on Theorem \ref{thm:OVV} (and its analogue for other Norm varieties associated to symbols with $\mu_\ell$-coefficients).


\subsection{The quadric surface bundle of Hassett--Pirutka--Tschinkel} \label{subsec:HPT}

For any smooth quadric surface $Q\subset \CP^3_K$ over a field $K$, the kernel of the pullback map
$$
H^2(K,\mu_2^{\otimes 2}) \longrightarrow H^2(K(Q),\mu_2^{\otimes 2})
$$
is completely described by Arason in \cite{arason}, see e.g.\ \cite[Theorem 3.10]{Pirutka}.
Moreover, one can use the Hochschild--Serre spectral sequence to show that the image of the above map is given by $H^2_{nr}(K(Q)/K,\mu_2^{\otimes 2})$.
In \cite[Theorem 3.17]{Pirutka}, Pirutka uses these ingredients to give a general formula for the unramified cohomology $H^2_{nr}(\C(Y)/\C,\mu_2^{\otimes 2})$,  where $Y$ is a variety over $\C$ that admits a fibration $f:Y\to \CP^2_\C$ whose generic fibre is a  smooth quadric surface over $\C(\CP^2)$.
This general formula has the following beautiful consequence, see \cite[Proposition 11]{HPT}.

\begin{proposition}[\cite{HPT}] \label{prop:HPT}
Let $g:=x_0^2+x_1^2+x_2^2-2(x_0x_1+x_0x_2+x_1x_2)$ be the equation of the conic in $\CP^2_{\C}$ that is tangent to the coordinate lines $x_i=0$ for $i=0,1,2$ and consider the bidegree $(2,2)$ hypersurface $Y\subset \CP^2_\C\times \CP^3_\C$, given by the equation
$$
Y:=\{ g(x_0,x_1,x_2) \cdot z_0^2+x_0x_1\cdot z_1^2+x_0x_2\cdot z_2^2+x_1x_2\cdot z_3^2=0\} \subset \CP^2_\C \times \CP^3_\C
.$$
If $f:Y\to \CP^2_\C$ denotes the natural morphism that is induced by the first projection, then
$$
0\neq f^\ast \left( \frac{x_1}{x_0},\frac{x_2}{x_0} \right) \in H^2_{nr}(\C(Y)/\C,\mu_2^{\otimes 2}) .
$$
\end{proposition} 

Even though the above result is formulated over $\C$,  the proof remains valid over any algebraically closed field of characteristic different from $2$.

One of the main differences between the examples in Section \ref{subsec:CTO} and \ref{subsec:HPT} is that the generic fibre of the quadric surface bundle of Hassett--Pirutka--Tschinkel in Proposition \ref{prop:HPT} is not (stably birational to) a Pfister quadric.
On the other hand, a common feature of both results is that they rely on the fact that the kernel of the pullback map
$$
H^n(K,\mu_2^{\otimes n})\longrightarrow H^n(K(Q),\mu_2^{\otimes n})
$$
is known in both cases: when $Q$ is a Pfister quadric or an arbitrary quadric surface.  
The kernel of the above map is not known for general quadrics, which yields a nontrivial obstacle when trying to generalize the result of Hassett--Pirutka--Tschinkel from Proposition \ref{prop:HPT} to higher dimensions.

\subsection{Generalization} \label{subsec:Sch}
The following generalization of Proposition \ref{prop:HPT} was discovered in \cite{Sch-JAMS} and \cite{Sch-torsion}.

\begin{theorem} \label{thm:examples}
Let $k$ be an algebraically closed field and let $m$ be a positive integer that is invertible in $k$.
Assume that there is an element $t\in k$ which is transcendental over the prime field of $k$.\footnote{If $k$ has characteristic 0, then $t$ may also be chosen to be a prime number that is coprime to $m$}
For $d:=m\cdot \lceil\frac{n+1}{m}\rceil$, consider the polynomial
$$
g:=t\cdot \left(  \sum_{i=0}^n x_i\right) ^d+x_0^{d-n}x_1x_2\dots x_n \in k[x_0,x_1,\dots ,x_{n}]
$$
and the bidegree $(d,m)$ hypersurface $Y\subset \CP^n\times \CP^{2^n-1}$, given by
$$
Y:= \left\lbrace  g(x_0,x_1,\dots ,x_n) \cdot z_0^m+\sum_{\substack{\epsilon\in \{0,1\}^n\\ \epsilon\neq 0}} x_0^{d-\sum_{i=1}^n \epsilon_i}x_1^{\epsilon_1}x_2^{\epsilon_2}\cdots x_n^{\epsilon_n}\cdot z_{\rho(\epsilon)}^m =0 \right\rbrace  \subset \CP^n_k \times \CP^{2^n-1}_k ,
$$
where  $\rho:\{0,1\}^n\to \{0,1,2,\dots ,2^n-1\}$ denotes the bijection $\rho(\epsilon)=\sum_{i=0}^{n-1}\epsilon_i\cdot 2 ^i $.
If $f:Y\to \CP_k^{n}$ denotes the morphism induced by the first projection, then the class
\begin{align} \label{eq:f*alpha}
0\neq f^\ast \left(\frac{x_1}{x_0},\frac{x_2}{x_0},\dots ,\frac{x_n}{x_0} \right)\in H^n_{nr}(k(Y)/k,\mu_m^{\otimes n}) 
\end{align}
has order $m$ and is unramified over $k$. 
\end{theorem}

A formal analogy between the example in the above theorem and that in Proposition \ref{prop:HPT} is as follows: the equation $g$ in Proposition \ref{prop:HPT} defines a conic that is tangent to the three coordinate lines in $\CP^2_k$ and so $g$ restricts to squares on the coordinate lines. Similarly, the equation $g$ in Theorem \ref{thm:examples} restricts to a $d$-th power (and hence to an $m$-th power because $m\mid d$) on each coordinate hypersurface $\{x_i=0\}\subset \CP^n_k$.

The starting point of Theorem \ref{thm:examples} is the following result, which generalizes one part of Theorem \ref{thm:OVV} from $m=2$ to arbitrary $m\geq 2$, see \cite[Corollary 4.2]{Sch-torsion}.

\begin{lemma} \label{lem:Fermat-Pfister:vanishing}
Let $K$ be a field and let $m$ be a positive integer that is invertible in $K$.
Let $a_1,\dots ,a_n\in K^\ast$ be invertible elements with associated symbol $\alpha=(a_1,\dots ,a_n)\in H^n(K,\mu_m^{\otimes n})$.
Consider further the hypersurface
$$
F_\alpha:=\left\lbrace  \sum_{ \epsilon\in \{0,1\}^n }  (-a_1)^{\epsilon_1}(-a_2)^{\epsilon_2}\cdots (-a_n)^{\epsilon_n}\cdot z_{\rho(\epsilon)}^m =0 \right\rbrace  \subset  \CP^{2^n-1}_K
$$
with structure morphism $f:F_{\alpha}\to \Spec K$.
Then
$$
f^\ast \alpha=0\in H^n(K(F_\alpha),\mu_m^{\otimes n}) .
$$ 
\end{lemma} 

\begin{proof}[Proof of Theorem \ref{thm:examples}]
The case $m=2$ follows from \cite[Propositions 5.1 and 6.1]{Sch-JAMS}, and the general case of arbitrary $m\geq 2$ follows from  \cite[Proposition 4.1]{Sch-torsion} and \cite[Theorem 5.3]{Sch-torsion}.
We give (almost all) the details of the argument in what follows.

Let
$$
\alpha:= \left(\frac{x_1}{x_0},\frac{x_2}{x_0},\dots ,\frac{x_n}{x_0} \right)\in H^n(k(\CP^n),\mu_m^{\otimes n}).
$$
Then we need to show the following two properties:  
\begin{enumerate}[(a)]
\item $f^\ast \alpha$ is unramified over $k$;\label{item:unramified}
\item  $f^\ast \alpha$  has order $m$.\label{item:order-m}
\end{enumerate}
Note that these properties are opposing to each other, as (\ref{item:unramified}) amounts to a vanishing result (all residues of the class in (\ref{eq:f*alpha})  vanish), while (\ref{item:order-m}) amounts to a non-vanishing result.  

To prove (\ref{item:order-m}), let $F\subset k$ be the algebraic closure of the prime field of $k$.
Then $Y$ is defined over $F[t]$ and so we can consider its degeneration $Y_0$ modulo $t$, which is a hypersurface $Y_0\subset \CP^n_F\times \CP^{2^{n}-1}_F$ with projection $f_0:Y_0\to \CP^n_F$.
For a contradiction, assume that there is an integer $e\in \{1,2,\dots ,m-1\}$ with $e\cdot f^\ast \alpha=0$.
It is not hard to show that this implies the following for the specialization where $t=0$:
$$
e\cdot f_0^\ast \alpha=0\in H^n(F(Y_0) ,\mu_m^{\otimes n}) .
$$
(Here by slight abuse of notation, we use that $\alpha\in H^n(F(\CP^n),\mu_m^{\otimes n})$.)
However, our construction implies that the generic fibre of $f_0:Y_0\to \CP^n_F$ has a $F$-rational point $y_0\in Y_0$ in its smooth locus. The class $e\cdot f_0^\ast \alpha$ can be restricted to this point and so
$$
e\cdot f_0^\ast \alpha|_{y_0}=0\in H^n(\kappa(y_0) ,\mu_m^{\otimes n}).
$$
On the other hand, $\kappa(y_0)\cong F(\CP^n)$ and the composition
$$
H^n(F(\CP^n),\mu_m^{\otimes n})\longrightarrow H^n(\kappa(y_0) ,\mu_m^{\otimes n})\cong H^n(F(\CP^n),\mu_m^{\otimes n})
$$
given by pullback via $f_0$ and restriction to $y_0$ is the identity.
Hence,
$$
e\cdot \alpha=0\in H^n(F(\CP^n) ,\mu_m^{\otimes n}) ,
$$
which is false as one may check by induction on $n$ by taking residues along $x_{n}=0$.
This proves (\ref{item:order-m}).

To prove (\ref{item:unramified}), let $x\in Y'$ be a codimension 1 point of a normal birational model of $Y$.
We may assume that there is a birational morphism $Y'\to Y$ and so $f$ induces a morphism $f':Y'\to \CP^n$.
We then need to show that $\del_x ({f'}^\ast\alpha)=0$, where
$$
\alpha:= \left(\frac{x_1}{x_0},\frac{x_2}{x_0},\dots ,\frac{x_n}{x_0} \right)\in H^n(k(\CP^n),\mu_m^{\otimes n}).
$$
This vanishing is obvious, unless $f'(x)$ is contained in the union of hyperplanes $\{x_0x_1\dots x_n=0\}\subset \CP^n_k$, which is the ramification locus of $\alpha$ on $\CP^n_k$.
It thus suffices to deal with the case where $f'(x)\in \{x_0x_1\dots x_n=0\}$.
The proof then splits up into two cases, as follows.

\textbf{Case 1.} $f'(x)\notin \{g=0\}$.

In this case, $g$ is a nontrivial $m$-th power in the residue field of the local ring $\mathcal O_{Y',x}$ and so it becomes an $m$-th power in the completion $\widehat{\mathcal O_{Y',x}}$.
The residue $\del_x:H^n(k(Y),\mu_m^{\otimes n})\to H^{n-1}(\kappa(x),\mu_m^{\otimes n-1})$ factors through 
$$
H^n(k(Y),\mu_m^{\otimes n})\longrightarrow H^n(\Frac \widehat{\mathcal O_{Y',x}},\mu_m^{\otimes n})
$$
and the image of $f^\ast \alpha$ via the above map vanishes by Lemma \ref{lem:Fermat-Pfister:vanishing}, so that $\del_x ({f'}^\ast\alpha)=0$ follows, as we want.

\textbf{Case 2.}  $f'(x)\in \{g=0\}$.

The main point about this case is that ${g=0}$ meets each strata of the union of the hyperplanes ${x_i=0}$ dimensionally transversely.
In particular, $f'(x)\in \{g=0\}$ implies that the number $c$ of coordinate functions $x_i$ that vanish at $f'(x)$ is strictly smaller than the codimension of $f'(x)$ in $\CP^n_k$:
$$
c< \codim_{\CP^n_k}(f'(x)).
$$
It then follows from Lemma \ref{lem:residue} that the valuation $\mu$ on $k(\CP^n)$ that is induced by restricting the valuation on $k(Y)$ that is given by the codimension 1 point $x\in Y'$ satisfies
$$
\del_\mu\alpha=\alpha'\cup \beta ,
$$
for some $\beta \in H^{c-1}(\kappa(\mu),\mu_m^{\otimes c-1})$ and $\alpha'\in H^{n-c}(\kappa(\mu),\mu_m^{\otimes n-c})$ with
$$
\alpha'\in \im(H^{n-c}(\kappa(f'(x)),\mu_m^{\otimes n-c})\longrightarrow H^{n-c}(\kappa(\mu),\mu_m^{\otimes n-c})) .
$$
Since $k$ is algebraically closed, $\kappa(f'(x))$ has cohomological dimension 
$$
n-\codim_{\CP^n_k}(f'(x))<n-c,
$$ 
see (\ref{eq:coho-dimension}).
Hence, $H^{n-c}(\kappa(f'(x)),\mu_m^{\otimes n-c})=0$, which implies $\alpha'=0$.
In particular, $\del_\mu\alpha=0$ and so $\del_x ({f'}^\ast\alpha)=0$ follows from (\ref{diag:del-pullbacks}). 
This completes the proof of Theorem \ref{thm:examples}.
\end{proof}

\section{Vanishing result and applications} \label{sec:vanishing}

The examples discussed in Sections \ref{subsec:CTO}, \ref{subsec:HPT} and \ref{subsec:Sch} all have the following feature in common:
Up to birational equivalence, there is a morphism $f:Y\to \CP^n_k$ whose generic fibre is smooth  and a class $\alpha\in H^n(k(\CP^n),\mu_m^{\otimes n})$ such that $f^\ast \alpha$ is nonzero and unramified.
One of the main discoveries in \cite{Sch-duke,Sch-ANT,Sch-JAMS,Sch-torsion} was the observation that all these examples have the following property in common:
for any smooth variety $Y'$ with a generically finite dominant morphism $\tau:Y'\to Y$, and for any point $x\in Y'$ with $\tau(x)\in Y^{\sing}$, 
 the restriction of $\tau^\ast f^\ast \alpha$ to $x$ vanishes:
$$
(\tau^\ast f^\ast \alpha)|_{x}=0\in H^n(\kappa(x),\mu_m^{\otimes n}) .
$$ 
This vanishing result is crucial, as it allows to apply Theorem \ref{thm:specialization-schreieder} without resolving the singularities of the special fibre (or even the whole family) of the degeneration.
The intuition behind this result is as follows:
\begin{itemize}
\item since the generic fibre of $f$ is smooth, it suffices to show $(\tau^\ast f^\ast \alpha)|_{x}=0$ whenever $f(\tau(x))\in \CP^n_k$ is not the generic point of $\CP^n_k$;
\item if $f(\tau(x))$ is not contained in the ramification locus of $\alpha$, then the restriction $(\tau^\ast f^\ast \alpha)|_{x}$ factors via the restriction of $\alpha$ to the point $f(\tau(x))\in \CP^n_k$.
Since the latter is not the generic point, it has codimension at least 1 and so its cohomological dimension is at most $n-1$, which shows that $\alpha|_{f(\tau(x))}=0$ and so $(\tau^\ast f^\ast \alpha)|_{x}=0$ as we want.
\item  if $f(\tau(x))$ is  contained in the ramification locus of $\alpha$, then the intuition is as follows: first note that $\tau^\ast f^\ast \alpha$ is unramified over $k$ by Proposition \ref{prop:functoriality}, because $f^\ast \alpha$ is unramified and $\tau$ is generically finite.
On the other hand, $\alpha$ is by assumption ramified locally around the point $f(\tau(x))$ and so the most natural reason for the fact that $\tau^\ast f^\ast \alpha$ is unramified over $k$ would be that this class in fact vanishes \'etale locally around $x$, so that it can be extended trivially across $x$.
In  \cite{Sch-duke,Sch-ANT,Sch-JAMS,Sch-torsion}  exactly this phenomenon is observed for all the examples mentioned in  Sections \ref{subsec:CTO}, \ref{subsec:HPT} and \ref{subsec:Sch}.
\end{itemize}

In the case where the generic fibre of $f$ is a smooth quadric, the following general vanishing result, which gives some evidence for the above intuition and which makes it very easy to apply Theorem \ref{thm:specialization-schreieder} in many situations, is proven in \cite{Sch-JAMS}.

\begin{theorem} \label{thm:vanishing}
Let $f:Y\to S$ be a surjective morphism of proper varieties over an algebraically closed field $k$ with $\operatorname{char}(k)\neq 2$ whose generic fibre is birational to a smooth quadric over $k(S)$.
Assume that there is a class $\alpha\in H^n(k(S),\mu_2^{\otimes n})$ with $f^\ast \alpha\in H^n_{nr}(k(Y)/k,\mu_2^{\otimes n})$, where $n=\dim(S)$. 

Then for any dominant generically finite morphism $\tau:Y'\to Y$ of varieties with $Y'$ smooth over $k$ and for any (scheme) point $x\in Y'$ which does not map to the generic point of $S$ via $f\circ \tau$, we have   $(\tau^\ast f^\ast \alpha)|_x=0\in H^n(\kappa(x),\mu_2^{\otimes n})$.
\end{theorem}

Combining these vanishing results with the examples in Sections \ref{subsec:CTO}, \ref{subsec:HPT} and \ref{subsec:Sch}, one deduces for instance the following from the specialization result in Theorem \ref{thm:specialization-schreieder}:
\begin{itemize}
\item A very general hypersurface $X\subset \CP^2_\C\times \CP^3_\C$ of bidegree $(2,2)$ is not stably rational \cite{HPT}. Since the smooth bidegree $(2,2)$ hypersurfaces in $\CP^2_\C\times \CP^3_\C$ that are known to be rational can be shown to be dense in moduli, this showed in particular that rationality is not an open nor a closed condition in smooth projective families, see \cite{HPT}.
\item Generalizations of the aforementioned result to all standard quadric surface bundles over $\CP^2$, different from Verra fourfolds and cubic fourfolds containing a plane (see \cite{Sch-ANT,Paulsen}), as well as to quadric bundles of arbitrary fibre dimension, see \cite{Sch-duke}.
\item A very general hypersurface $X\subset \CP^{N+1}_k$ of dimension $N\geq 3$ and degree $d\geq \log_2N+2$ (resp.\ $d\geq \log_2N+3$) is not stably rational if $k$ is an uncountable field of characteristic different from $2$ (resp.\ equal to $2$), see \cite{Sch-JAMS,Sch-torsion}.
This improved earlier bounds of Koll\'ar \cite{kollar} and Totaro \cite{totaro} that were linear (roughly $d\geq \frac{2}{3}N$).

The results in \cite{Sch-JAMS,Sch-torsion} involve  a degeneration of a very general hypersurface to a hypersurface $Z$ that has high multiplicity along a large-dimensional linear subspace $P$.
One then needs to pass to the blow-up $Y=Bl_PZ$ and one exploits that projection from $P$ induces a fibration structure $f:Y\to \CP^n$ where $n=\dim X-\dim P$.
%
%
While one can then use results \`a la Theorem \ref{thm:vanishing}  to get the desired vanishing condition (\ref{cond:vanishing}) from Theorem \ref{thm:specialization-schreieder} at points that do not dominate $\CP^n$ via $f$, this vanishing condition needs to be checked by hand on the exceptional divisor of the blowup $Bl_PZ\to Z$, which makes the argument somewhat subtle.
\end{itemize}

\begin{remark}
In dimension $N=5$, the logarithmic bound from \cite{Sch-JAMS} shows that very general hypersurfaces of degree at least five in $\CP^6_k$ are not stably rational.
In the case where $k$ has characteristic 0, this result was improved by Nicaise and Ottem \cite{NO}, who showed that very general quartic fivefolds are stably irrational over fields of characteristic 0.
Their result is achieved by an application of Theorem \ref{thm:NS-KT}, and it relies eventually on the stable irrationality of the quadric bundles of Hassett--Pirutka--Tschinkel from Section \ref{subsec:HPT}, whose discriminant locus has smaller degree than those of the corresponding generalizations in Theorem \ref{thm:examples} that have been used in (arbitrary dimension) in \cite{Sch-JAMS}. 
\end{remark}

\section{Open Problems}\label{sec:problem}

\subsection{Decompositions of the diagonal versus stable rationality}
Recall from Lemma \ref{lem:ratl->decofdiag} that a variety that is stably rational admits a decomposition of the diagonal.
It is natural to wonder whether the converse to this statement holds as well.
It is known that a smooth complex projective surface $X$ with $\CH_0(X)\cong \Z$ and without torsion in $H^2(X,\Z)$ admits a decomposition of the diagonal, see \cite[p.\ 1252, Remark (2)]{BS}, \cite[Corollary 1.10]{ACTP}, \cite[Corollary 2.2]{Voi-JEMS} or \cite{kahn}.
There are such surfaces that are of general type and so existence of a decomposition of the diagonal is in general not equivalent to stable rationality.
However, no such counterexample is known if we restrict to the class of rationally connected varieties.

\begin{question}\label{question:dec-of-diag}
Is there a rationally connected smooth complex projective variety which admits a decomposition of the diagonal, but which is not stably rational?
\end{question}

The above question is  already open for the Fermat cubic threefold, which admits a decomposition of the diagonal by a result of Colliot--Th\'el\`ene \cite{CT17}, while stable irrationality is unknown.
In fact, while any smooth cubic threefold in $\CP^4_\C$ is irrational  by a celebrated result of Clemens and Griffiths \cite{clemens-griffiths}, the question whether it is also stably irrational is open for any such cubic; cf.\ \cite{Voi-JEMS} for an interesting connection of this question to the integral Hodge conjecture on abelian varieties.

A natural approach to Question \ref{question:dec-of-diag} would be given by the obstruction for stable rationality in \cite{NS,KT} (see Theorem \ref{thm:NS-KT}), which might a priori not be sensitive to decompositions of the diagonal.

A related classical open question pointed out by one of the referees is as follows.
\begin{question}
Is there a variety $X$ over an algebraically closed field $k$ which is not stably rational but such that $X\times_k Y$ is rational for some variety $Y$ over $k$?
\end{question}

\subsection{Torsion orders and unirationality}

Let $X$ be a rationally chain connected projective variety over a field $k$.
Then its torsion order $\Tor(X)$ is finite, see Section \ref{subsec:Tor(X)}.
The class of rationally chain connected varieties is closed under several natural operations (e.g.\ taking products and taking quotients) and so it is natural to investigate how the torsion orders behave when performing these operations.

\begin{lemma}
Let $X$ and $Y$ be proper varieties over a field $k$ whose torsion orders are finite.
Then
$$
\Tor(X\times Y)\mid \Tor(X)\Tor(Y).
$$
\end{lemma}
\begin{proof}
The diagonal of $X\times Y$ corresponds to the product of the diagonals of $X$ and $Y$.
Since $\Tor(X)\Delta_X$ and $\Tor(Y)\Delta_Y$ admit decompositions as in (\ref{def:dec-of-diag}), we conclude that 
$$
\Tor(X)\Tor(Y)\Delta_{X\times Y}
$$
admits a similar decomposition and so $\Tor(X\times Y)\mid \Tor(X)\Tor(Y)$, as we want.
\end{proof}

\begin{lemma} \label{lem:Tor(S^nX)}
Let $X$ be a smooth projective variety over a field $k$ whose torsion order is finite.
For a positive integer $n$, we consider the symmetric product $S^nX=X^n/\Sym(n)$ of $X$.
Then
$$
  \Tor( S^nX )\mid n!\cdot \Tor(X)^n
$$
\end{lemma}
\begin{proof}
By Lemma \ref{lem:Tor(X)}, applied to the quotient map $X^n\to S^nX$, we have
$$
  \Tor( S^nX )\mid n!\cdot \Tor(X^n) 
$$
and so the claim follows from the previous lemma, which implies $  \Tor(X^n) \mid \Tor(X)^n$.
\end{proof}


\begin{question}
Let $X$ be a smooth complex projective variety with finite torsion order.
Is it possible to improve the estimate from Lemma \ref{lem:Tor(S^nX)} for the torsion order of $S^nX$?
\end{question}

More precisely, we may ask the following.

\begin{question} 
Is it true that for any smooth complex projective variety $X$ with finite torsion order and for any positive integer $n$, we have
$
  \Tor( S^nX )\mid   \Tor(X)^n 
$ (or maybe even $\Tor( S^nX )\mid   \Tor(X)$)?
\end{question}

In the opposite direction, it is natural to wonder about the following.

\begin{question} \label{question:unirational}
Is there a smooth complex projective variety $X$ that is rationally (chain) connected such that the prime factors of
$
  \Tor( S^nX ) 
$ are unbounded for $n\to \infty$?
\end{question}

A positive answer to Question \ref{question:unirational} would, by the following proposition, imply the existence of a rationally connected\footnote{For smooth complex projective varieties, rationally chain connectedness and rationally connectedness coincide, see \cite{kollar2}.} smooth complex projective variety that is not unirational, which is a longstanding open problem in the field.

\begin{proposition} \label{prop:unirational}
Let $X$ be a variety over a field $k$ and assume that there is a dominant rational map $f:\CP^{\dim X}_k\dashrightarrow X$.
Then for all $n\geq 1$, we have
$$
\Tor(S^nX)\mid \deg(f)^n.
$$
In particular, the prime factors of $\Tor(S^nX)$ are bounded for $n\to \infty$.
\end{proposition}
\begin{proof}
Taking the $n$-th symmetric power of $f$, we obtain a dominant rational map
$$
S^nf:S^n\CP_k^{\dim X}\dashrightarrow S^nX .
$$ 
The degree of this map may be identified to $\deg S^nf=\deg(f)^n$.
Hence, Lemma \ref{lem:Tor(X)} implies
$$
\Tor(S^nX)\mid \deg(f)^n\cdot \Tor(S^n\CP^{\dim X}) .
$$
On the other hand, $S^n\CP^{\dim X} $ is rational by an old result of Mattuck \cite{Mat68}.
Hence, $\Tor(S^n\CP^{\dim X})=1 $ by Lemma \ref{lem:ratl->decofdiag} and so the proposition follows. 
\end{proof}

\section*{Acknowledgements}  

Thanks to the organizers of the Schiermonnikoog conference on Rationality of Algebraic Varieties in spring 2019 for inviting me to write this survey.
Parts of this survey rely on lecture series that I have given in Moscow and Nancy in spring 2019.
The comments of two excellent referees significantly improved this text.


\end{document}